\documentclass[11pt]{amsart}
\usepackage{amsfonts, amsmath, amssymb, amscd, amsthm, graphicx, setspace, enumitem, color}
\usepackage{hyperref}
\usepackage{pinlabel}

\newtheorem{theorem}{Theorem} [section]
\newtheorem*{RivC}{Rivin's Conjecture}
\newtheorem*{SRivC}{Strengthened Rivin's Conjecture}
\newtheorem{lemma}[theorem]{Lemma}
\newtheorem{proposition}[theorem]{Proposition}
\newtheorem{corollary}[theorem]{Corollary}

\newtheorem{remark}[theorem]{Remark}

\theoremstyle{definition}
\newtheorem{definition}[theorem]{Definition}
\newtheorem{notation}[theorem]{Notation}
\newcommand{\gen}[1]{\langle#1\rangle}
\newcommand{\ol}[1]{\overline{#1}}
\newcommand{\prs}[2]{\gen{#1\parallel #2}}

\newcommand{\wh}[1]{\widehat{#1}}

\newcommand\R{{\mathbb R}}
\newcommand\N{{\mathbb N}}
\newcommand\Z{{\mathbb Z}}

\newcommand\Q{{\mathbb Q}}

\newcommand\conj{\sim}
\newcommand\subconj{_{cj}}
\newcommand\subcomm{_{cm}}
\newcommand\subpconj{_{pc}}

\newcommand\geol{\mathsf{Geo}}   % \Gamma geodesic language
\newcommand\cyc{\mathsf{Cyc}}    % geodesic growth series    %% use {\mathcal G}?
\newcommand\geocl{\mathsf{ConjGeo}}   % \widetilde{\Gamma} geodesic conjugacy language
\newcommand\geocpl{\mathsf{CycGeo}}   % \breve{\Gamma} geodesic cyclic conjugacy language

\renewcommand{\kappa }{\varkappa}
\newcommand\e{\varepsilon}
\newcommand{\h}{\hookrightarrow_{h}}

\renewcommand{\d}{{\rm d}}

\newcommand{\ga}{\Gamma}
\newcommand{\Lab}{\mathsf{Lab}}

\newcommand{\cF }{\mathcal F}
\newcommand{\cG }{\mathcal G}
\newcommand{\cH }{\mathcal H}
\newcommand{\cL }{\mathcal L}
\newcommand{\cM }{\mathcal M}

\newcommand{\cR }{\mathcal R}
\newcommand{\cS }{\mathcal S}
\newcommand{\cT }{\mathcal T}

\newcommand{\grate}{\mathbf{h}}

\newcommand{\nei}[3]{\mathcal{N}_{#3}^{#2}(#1)} 
\newcommand{\unpad}{\rm unpad}

\newcommand{\ball}{\mathbb{B}}
\newcommand{\diam}{{\rm diam}}
\def\coloneq{\mathrel{\mathop\mathchar"303A}\mkern-1.2mu=}

\begin{document}

\author{Yago Antol\'{i}n and Laura Ciobanu}
\title[Formal conjugacy growth in acylindrically hyperbolic groups]{Formal conjugacy growth in acylindrically hyperbolic groups}

\begin{abstract}
Rivin conjectured that the conjugacy growth series of a hyperbolic group is rational if and only if the group is virtually cyclic. 
Ciobanu, Hermiller, Holt and Rees proved  that the conjugacy growth series of a virtually cyclic group is rational.
Here we present the proof confirming the other direction of the conjecture, by showing that the conjugacy growth series of a non-elementary hyperbolic group is transcendental.  We also present and prove some variations of Rivin's conjecture for commensurability classes and primitive conjugacy classes.

We then explore Rivin's conjecture for finitely generated acylindrically hyperbolic groups and prove a formal language version of it, namely that no set of minimal length conjugacy representatives can be unambiguous context-free.
\bigskip


\end{abstract}
\maketitle

%%%%%%%%%%%%%%%%%%%%%%%%%%%%%%%%%%%%%%%%%%%%%%%%%%%%%%%%%%%%%%%%%%%%%%%%%%%%
\section{Introduction}\label{sec:intro}
%%%%%%%%%%%%%%%%%%%%%%%%%%%%%%%%%%%%%%%%%%%%%%%%%%%%%%%%%%%%%%%%%%%%%%%%%%%%
%%%%%%%%%%%%%%%%%%%%%%%%%%%%%%%%%%%%%%%%%%%%%%%%%%%%%%%%%%%%%%%%%%%%%%%%%%%%

For any $n\geq 0$, the \textit{conjugacy growth function} $\phi\subconj(n)$ of a 
finitely generated group counts the number of conjugacy classes in the ball of radius $n$. 
This function has recently been studied
by several authors for many important classes of groups (see the 
survey~\cite{gubasapir} and also ~\cite{bdc}, \cite{hullosin}) with the goal of 
determining whether it is polynomial or exponential, or to establish uniform 
conjugacy growth. 
Primitive conjugacy growth, the one that counts the number of conjugacy classes of primitive  elements 
(a primitive element is one which is not a proper power of another element) 
has  been studied for many decades, 
%from a geometric perspective, %for many decades, 
motivated by counting closed geodesics (up to free 
homotopy) on complete Riemannian manifolds; for example, Margulis \cite{margulis} 
proved that for a compact manifold $M$ of pinched negative curvature and 
exponential volume growth $e^{\grate t}$, where $\grate$ is the entropy of $M$, the number of 
primitive closed geodesics of period $\leq t$ is approximately $\frac{e^{\grate t}}
{ t}$. 
This formula gives, via quasi-isometry, good estimates for the number of 
primitive conjugacy classes  in the fundamental group of $M$.

The previously mentioned results study conjugacy growth from the asymptotic point of view. 
In this paper we 
study formal conjugacy growth, that is, the formal power series 
\begin{equation}\label{eq:growthseries}
\cG\subconj (z) = \sum_{n=0}^{\infty} \phi\subconj(n)z^n \in \Z[[z]].
\end{equation}
Notice that when studying formal growth, the algebraic complexity of function \eqref{eq:growthseries} (i.e. rational, algebraic or transcendental over $\Q(z)$) might depend on the choice of the generating set. Throughout the paper, we will assume that the generating sets generate the group as a monoid. In order to avoid working with non-symmetric metrics on Cayley graphs, we will also assume that the generating sets are inverse closed, although this last condition is not essential.

To our knowledge, Rivin (\cite{rivin04} and \cite{rivin}) was the first to study formal conjugacy growth for groups. He computed $\cG\subconj$ for non-abelian free groups with respect to the standard basis and showed that it
is not rational (see Section \ref{sec:series} for a definition). Ciobanu and Hermiller~\cite{ciobanuhermiller} obtained similar results for free products of finite groups different from $\Z/2\Z*\Z/2\Z$. 
 Rivin, based on his computations, made the following conjecture. 

\begin{RivC}\cite[Conjecture 13.1]{rivin}\label{conj:rivin}
Let $G$ be a word hyperbolic. The conjugacy growth series
$\cG\subconj$ is rational if and only if $G$ is virtually cyclic.
\end{RivC}

In \cite{CHHR}, Ciobanu, Hermiller, Holt and Rees proved one direction of Rivin's conjecture, namely that virtually cyclic groups have rational conjugacy growth series for any generating set.

In this paper we complete the proof of Rivin's Conjecture and give natural generalizations to primitive conjugacy growth and commensurability growth. 
As mentioned before, from the geometric point of view it is natural to study the primitive conjugacy growth $\phi\subpconj(n)$, i.e. the number of conjugacy classes of primitive elements of length at most $n$. From the algebraic point of view, on the other hand,
when the group has torsion it is more suitable to study commensurability classes. Recall that two elements $g$ and $h$ of a group $G$ are {\it commensurated} if there are $m,n\in \Z-\{0\}$ and $k\in G$ such that $k^{-1}g^mk=h^n$. Commensurability classes in acylindrically hyperbolic groups were used in  \cite{AMS} to decide if certain endomorphisms were inner automorphisms. Thus the number of commensurability classes $\phi\subcomm (n)$ in the ball of radius $n$ can be relevant for algorithms.  

Let $\cG\subpconj$ and $\cG\subcomm$ denote the growth series associated to $\phi\subpconj$ and $\phi\subcomm$, respectively. 
The first main result is the following.

\begin{theorem}\label{thm:rivinconj}
Let $G$ be a hyperbolic group. Then the growth series $\cG\subconj$, $\cG\subpconj$ and $\cG\subcomm$  with respect to any finite symmetric generating set are:
\begin{enumerate}
\item[{\rm (1)}] rational over $\mathbb{Q}(z)$, if $G$ is virtually cyclic.
\item[{\rm (2)}] transcendental over $\mathbb{Q}(z)$, if $G$ is not virtually cyclic.
\end{enumerate}
\end{theorem}

Our proof of Theorem \ref{thm:rivinconj} relies on 
understanding the asymptotics of $\phi\subconj$, $\phi\subpconj$ and $\phi\subcomm$ for non-elementary hyperbolic groups. A hyperbolic group is called {\it elementary} if it is virtually cyclic. 
Coornaert and Knieper (\cite{ckGAFA}, \cite{ckIJAC}) generalized the results of Margulis to the context of  hyperbolic groups, and provided bounds (the upper bound only in the torsion-free case) for the
growth function of conjugacy classes and primitive conjugacy classes in terms of the \emph{exponential growth rate} $\grate$ of $G$ with respect to the generating set $X$, where $\grate =\limsup_{n\rightarrow \infty} \log \sqrt[n]{|\ball_X(n)|}$ and $\ball_X(n)$ is the ball of radius $n$. In this paper we extend the results of Coornaert and Knieper to all conjugacy classes and all commensurability classes, and we prove 
\begin{theorem}\label{thm:cgrate}
Let $G$ be a non-elementary word hyperbolic group and $X$ any finite symmetric  generating set. There are positive constants $A, B$ and $n_0$ such that
$$A\frac{e^{n\grate }}{n} \leq  \phi\subcomm(n)\leq \phi\subpconj(n) \leq \phi\subconj(n) \leq B \frac{e^{n\grate}}{n}$$ for all $n \geq n_0$.

\end{theorem}

In general, one can define the \emph{exponential growth rate} of any positive function $f$ to be the quantity $\limsup_{n\rightarrow \infty} (\log \sqrt[n]{f(n)})$. A direct consequence of the previous theorem is
\begin{corollary} \label{growth_rates}
Let $G$ be hyperbolic and $X$ any finite generating set.
The exponential growth rate $\grate$ of $G$ and the exponential growth rates of $\phi\subconj$, $\phi\subcomm$ and $\phi\subpconj$ are equal.
\end{corollary}

The proof of Theorem \ref{thm:cgrate} requires two key ingredients: the first is Patterson-Sullivan theory, which is used to obtain bounds for the growth of  hyperbolic groups; 
the second is the fact that in hyperbolic groups ``conjugacy $\cong$ cyclic permutation'', i.e. two elements are conjugate if and only if some of their cyclic permutations are conjugated by an element of universally bounded length.

Our second main result concerns the Rivin conjecture for {\it acylindrically  hyperbolic groups}, that is, groups acting non-elementary and acylindrically by isometries on a hyperbolic space (more details in Section \ref{sec:AH}). The class of acylindrically hyperbolic groups contains non-elementary
hyperbolic and relatively hyperbolic groups, all but finitely many mapping class
groups of punctured closed surfaces, $\mathrm{Out}(F_n)$ for $n \geq 2$, directly indecomposable right-angled
Artin groups, 1-relator groups with at least 3 generators, most 3-manifold groups, $C'(\frac{1}{6})$ small cancellation groups and many
other examples (see \cite{DGO,GS,MO}). 

Based on our results we propose the following natural extension of Rivin's conjecture.
\begin{SRivC}
Let $G$ be a finitely generated acylindrically hyperbolic group.
Then $\cG\subconj$, $\cG\subpconj$ and $\cG\subcomm$ are transcendental.
\end{SRivC}

The action of acylindrically hyperbolic groups on hyperbolic spaces is typically
neither proper, nor co-compact; hence it is not clear how to define Patterson-Sullivan measures on the boundary (recently, \cite{Yang} has made progress in the case of 
relatively hyperbolic groups). Therefore a result like Theorem \ref{thm:cgrate} for finitely generated acylindrically hyperbolic groups seems out of reach with the current techniques. This is why we approach the above conjecture from the point of view of formal languages.

Recall that {\it a language} $\cL$ is a set of words over some finite {\it alphabet} $A$, that is, $\cL$ is a subset of $A^*$, the free monoid generated by $A$. Languages have been categorized  into several classes according to their complexity, the simplest ones being regular and context-free, with the class of unambiguous context-free languages strictly containing all regular languages and being contained in the set of all context-free ones. For these low-level languages we can match the computational complexity with an algebraic characterization as follows: the growth series of a regular language is rational, and 
 the growth series of an unambiguous context-free language (the definition is not necessary in this paper, see \cite{hu} for details)  is algebraic. 

\begin{theorem}[Chomsky-Sch\"utzemberger]\label{thm:ChomskyS} If $\cL\subseteq A^*$ is
unambiguous context-free, then $\cG_\cL(z)= \sum_{n=0}^{\infty} \sharp \{W\in \cL \mid \ell(W)\leq n\} z^{n}\in \Z[[z]]$ is algebraic over $\Q(z)$.
\end{theorem}

Chomsky-Sch\"utzemberger's theorem would imply, if the Strengthened Rivin's conjecture were confirmed, that no language of minimal length conjugacy/conjugacy primitive/commensurability representatives can be unambiguous context-free. This is exactly our second main result.

\begin{theorem}\label{thm:AHregular}
Let $G$ be a finitely generated acylindrically hyperbolic group, $X$ any finite symmetric generating set, and $\cL\subconj/\cL\subpconj/\cL\subcomm$ a subset of $X^*$ containing  exactly one minimal
length representative of each conjugacy/primitive conjugacy/commensurability  class. Then $\cL\subconj/\cL\subpconj/\cL\subcomm$ is not unambiguous context-free. In particular, such a language cannot be regular.
\end{theorem}

The proof of Theorem \ref{thm:AHregular} combines closure properties of formal languages, which will be discussed in Section \ref{sec:languages}, together with the ``conjugacy $\cong$ cyclic permutation" phenomenon for conjugacy classes of hyperbolically embedded subgroups (Theorem \ref{thm:poison}) proved in Section \ref{sec:AH}.

%%%%%%%%%%%%%%%%%%%%%%%%%%%%%%%%
%% Transcendental
%%%%%%%%%%%%%%%%%%%%%%%%%%%%%%%%
\section{Transcendence of growth series and the proof of Theorem \ref{thm:rivinconj}}\label{sec:series}

 We say that a formal power series
$f(z) \in \Z[[z]]$ is {\em rational} if there exist nonzero polynomials $p(z), q(z) \in \Z[z]$  such that $f(z)=\frac{p(z)}{q(z)}$; more generally, $f(z)$ is {\em algebraic} over $\Q(z)$ if there exists a nontrivial polynomial $p(z,u)\in\Q(z,u)$ such that $p(z,f(z))=0.$ If $f(z)$ is not algebraic over $\Q(z)$, we say that $f(z)$ is {\em transcendental} over $\Q(z)$.

Any language $\cL$ over $X$ gives rise to a
{\em strict growth function} $\sigma_\cL:\N \cup \{0\} \rightarrow \N \cup \{0\}$
defined by  $$\sigma_\cL(n) := |\{W \in \cL \mid \ell(W) = n\}|,$$
and a {\em cumulative growth function} $\phi_\cL:\N \cup \{0\} \rightarrow \N \cup \{0\}$ 
defined by  $$\phi_\cL(n) := |\{W \in \cL \mid \ell(W) \leq n\}|.$$
These, in turn, give rise to the strict growth series 
$\sum_{i=0}^\infty \sigma_{\cL}(n) z^n$ 
and cumulative growth series 
$\sum_{i=0}^\infty \phi_{\cL}(n) z^n$ 
of the language $\cL$.
In the Introduction, Theorem \ref{thm:ChomskyS} is stated for the cumulative growth series of the language, while in the literature such results may be stated
for the strict growth series. We observe that these two formulations are equivalent, since from the relation 
$$ \sum_{i=0}^\infty \sigma_{\cL}(n)z^n =(1-z)\sum_{i=0}^\infty \phi_{\cL}(n) z^n $$
the strict growth series is rational (algebraic) if and only if the cumulative growth series is rational (algebraic).

We now deduce Theorem \ref{thm:rivinconj} from Theorem \ref{thm:cgrate}.
 Our proof relies on the lemma below, which describes the asymptotics of the coefficients of an algebraic growth series.

\begin{lemma}[\cite{Flajolet}, Theorem D] \label{lem:algebraic_coefficients}

If  $\sum_{i=0}^\infty b_nz^n$ is the series expansion of an algebraic function that is analytic at the origin, then there exist algebraic numbers $\alpha_j$ and $c_j$, $|\alpha_j|=1$, $0\leq j \leq k$, where $k$ is some positive integer, and a positive algebraic number $\lambda$  such that 
\begin{equation}\label{alg}
b_n = \sum_{i=0}^k c_i n^p \alpha_i^n \lambda^n + O(n^q \lambda^n),
\end{equation}
where $p \in \mathbb{Q} \setminus \{-1, -2, \dots\}$ and $q<p$.
\end{lemma}

\begin{proof}[Proof of Theorem \ref{thm:rivinconj}]
First observe that if $G$ is virtually cyclic, the result for $\cG\subconj$ was proved in \cite[Theorem 4.2]{CHHR}, and $\cG\subpconj$ and $\cG\subcomm$ are polynomials, as in a virtually cyclic group there are finitely many primitive conjugacy classes and commensurability classes.

So, assume that $G$ is not virtually cyclic and let $b_n=\phi\subconj(n)$ or $b_n=\phi\subpconj(n)$ or $b_n=\phi\subcomm(n)$. By Theorem \ref{thm:cgrate} the sequence $b_n$ satisfies $A\leq b_n \frac{n}{e^{n\grate}} \leq B$ for $n\geq n_0$.
By the root criterion for convergence, $\sum_{n=0}^\infty b_n z^n$ converges for $|z|<1/e^\grate$ and therefore is analytic at the origin.
Suppose moreover that $\sum_{n=0}^\infty b_n z^n$ is an algebraic function; then by Lemma \ref{lem:algebraic_coefficients} we can assume that $b_n$ has the form (\ref{alg}).

Since $A\leq b_n \frac{n}{e^{n\grate}} \leq B$ for $n\geq n_0$, we have that  $\lambda = e^{\grate}$ and $p=-1$.  
Thus by Lemma \ref{lem:algebraic_coefficients} the  growth series $\sum b_n z^n$ cannot be algebraic, so it is transcendental over $\mathbb{Q}(z)$.
\end{proof}
\begin{remark}\label{rem:quasicommgrowth}
The proof above, together with the assumption that $\phi\subpconj (n)$ satisfies the inequalities in Theorem \ref{thm:cgrate}, shows that if there is a constant $K>0$ and a function $\phi_\cL(n)$ such that $K^{-1} \phi\subpconj (n)\leq \phi_\cL(n)\leq K \phi\subpconj(n)$ for all $n$ greater than some $n_0$, then $\cG_\cL= \sum_{n\geq 0} \phi_\cL(n) z^n$ is transcendental over $\Q(z)$.
\end{remark}

\section{Permute and conjugate: the Bounded Conjugacy Diagrams (BCD) property}
One of the main ideas behind the proofs in this paper is that in a hyperbolic setting ``conjugacy $\cong$ cyclic permutation". In this section we
make precise this property, which we call the BCD property.
We start by fixing some notation used throughout the paper.

As already mentioned, by {\it a generating set of a group $G$} we mean a
set $X$ that generates $G$ as a monoid and that is closed under taking inverses. Every element of $G$
can be expressed as a word over $X$; for the sake of convenience we will identify a {\it word} $W\in X^*$ with the element
it represents in $G$. Given $U,V\in X^*$, we use $U\equiv V$ to denote word equality and $U =_G V$ to denote equality between the group elements represented by these words.
For a word $U$ in $X^*$, $\ell(U) (= \ell_X(U))$ denotes its length. For $g \in G$,
$|g|_X:= \min \{\ell(U) \mid U=_G g\}$. Our identification of a word with the group element it represents allows for the notation $|U|_X=\min\{\ell(V) \mid V=_G U\}$.
We denote  the closed ball of radius $n$ of $G$ with respect to $X$ by
$\ball_X(n) \coloneq \{g\in G \mid |g|_X \leq n\}.$

Our notation for paths in $\ga(G,X)$, the Cayley graph of $G$ with respect to $X$, is as follows. By $L[0,n]$ we denote the unlabelled graph with vertex set $\{0,1,2,\dots, n\}$ and edges joining $i$ 
to $i+1$ for $i=0,\dots, n-1$.  A {\it path $p$ of length $n$} in $\ga(G,X)$ is a
combinatorial graph map $p\colon L[0,n]\to \ga(G,X)$.
 In particular, $p(i)$ is a vertex of $\ga(G,X)$, 
$p(0)$ will be denoted by $p_-$ and $p(n)$ by  $p_+$. 
Let $\ell(p)$ be the length of a path $p$ and $\Lab(p)$ its label, i.e. the word read while traversing the edges of the path. A path $p$ in $\ga(G,X)$ is {\it geodesic} if 
$\ell(p)$ is minimal among the lengths of all paths $q$ with same endpoints.  
Let  $\lambda \geq 1$ and $\e\geq 0$. 
A path $p$ is a $(\lambda, \e)$-{\it quasi-geodesic} 
if for any subpath $q$ of $p$ we have 
$$\ell ( q )\leq \lambda \d(q_{-},q_{+})+\e.$$Let $K>0$. 
A path $p$ is a {\it $K$-local (quasi-)geodesic} 
if every subpath of $p$ of length less than or equal to $K$ is (quasi-)geodesic.

A word $U$ is a geodesic, cyclic geodesic, quasi-geodesic, etc. if any path in the Cayley graph $\ga(G,X)$ labelled  by $U$ has this property.
By a {\it a cyclic permutation of $W$} we mean a cyclic shift of the letters of $W$.
A word is a {\it cyclic ($(\lambda,\e)$-quasi-)geodesic} if all its cyclic permutations are ($(\lambda,\e)$-quasi-)geodesics. 
We use $\geol(G,X)$ and $\geocpl(G,X)$ to denote the set of all geodesic and cyclic geodesic words, respectively, and denote by $\geocl(G,X)$ the set of all words of minimal length in their conjugacy class. Note that $$\geocl(G,X)\subseteq \geocpl(G,X)\subseteq \geol(G,X).$$

Let $[g]\subconj$ denote the conjugacy class of $g \in G$.
\begin{definition}\label{def:BCD}
Let $\lambda \geq 1$ and $\e\geq 0$. A Cayley graph $\ga(G,X)$ satisfies the $(\lambda,\e)$-Bounded Conjugacy Diagram ($(\lambda,\e)$-BCD) property 
if  there is a constant $D=D(\lambda,\e)$
so that for any cyclic $(\lambda,\e)$-quasi-geodesics $U,V$, with  $V\in [U]\subconj$,
either 
\begin{enumerate}
\item[(1)] $\max\{\ell(U),\ell(V)\}\leq D$, or

\item[(2)] there exist $k\in \ball_X(D)$  and cyclic permutations $U'$ and $V'$ of $U$ and $V$ such that  $U'=_G k^{-1} V' k$.
\end{enumerate}
\end{definition}
We say that $\ga(G,X)$ has the BCD property if it has the $(\lambda,\e)$-BCD property for some $\lambda\geq 1$ and $\e\geq 0$. We will often use the fact that hyperbolic Cayley graphs have the BCD property (we allow infinite generating sets, see \cite[III.$\Gamma$. Lemma 2.9]{BridsonHaefliger}). 

Other examples with the BCD property can be found in \cite{AC3}, where it is shown that Cayley graphs of groups hyperbolic relative to abelian 
subgroups have the $(1,0)$-BCD property with respect to certain generating sets. 
In \cite{AS15} an analogue property to BCD is proved; there the input words $U$ and $V$ are quasi-geodesics
and not necessarily cyclic quasi-geodesics.

\section{Proof of Theorem \ref{thm:cgrate}: the conjugacy growth function bounds}

In this section we determine the upper and lower bounds for the  growth functions  of Theorem \ref{thm:cgrate}. We rely on the bounds obtained by Coornaert and Knieper in \cite{ckIJAC} and \cite{ckGAFA} for the growth of primitive conjugacy classes. 

\begin{definition}
Let $G$ be a group and $g\in G$ be an  element of infinite order. Let 
$$E_G^+(g)=\{h\in G \mid h^{-1} g^m h=g^m \text{ for some }m\in \Z-\{0\} \}.$$

We say that $g \in G$ is {\it primitive} if $g$ has infinite order and there is a finite subgroup $F\leqslant G$ such that $E_G^+(g) = \langle g \rangle \cdot F$ and $E_G^+(g)\cong \langle g \rangle \times F$.
\end{definition}

Notice that when $G$ is torsion-free, if $g$ is primitive then $g$ has infinite order and  $E_G^+(g)=\gen{g}$. Since $C_G(g)$, the centralizer of $g$ in $G$, is a subgroup of $E_G^+(g)$, it follows that $\gen{g}=C_G(g)$ and $g$ generates its own centralizer. In particular, $g$ cannot be a non-trivial power of another element in the group, which is the standard definition of primitivity. In the case when torsion is present, Proposition \ref{prop:prim}~(b) provides a characterization of primitives as non-powers in general: $g$ is primitive if its only `roots' are $g$ and $g^{-1}$.

The following lower bound is implicit in the proof of \cite[Thm. 1.1]{ckGAFA}.

\begin{theorem}\cite[Thm. 1.1]{ckGAFA}\label{prim_bounds}
Let $G$ be a  non-elementary hyperbolic group. 
There are positive constants $A$ and $n_0$ such that
$$A\frac{e^{\grate n}}{n} \leq \phi_{\cL\subpconj}(n) $$ for all $n \geq n_0$.
\end{theorem}

The goal of this section is twofold: first, we obtain the upper bound missing in Theorem \ref{prim_bounds}. This is also based on the work of Coornaert and Knieper \cite{ckIJAC}, where the upper bound is obtained in the torsion-free case.
Second, we prove that there is a constant $K$ such that in a given
commensurability class there are at most $K$ primitive elements. This will give the lower bound for commensurability classes.

For the rest of the section, we assume that $G$ is a non-elementary hyperbolic group and $X$ is a fixed symmetric generating set. We let 
\begin{equation}\label{eq:finite}
M=M(G):=\sup \{|F| \,\mid\, |F| < \infty, F \leq G \}
\end{equation}
 be the supremum of the  finite subgroup sizes in $G$.
By \cite[III.$\Gamma$.Theorem 3.2]{BridsonHaefliger}, $M$ is finite.

In the next paragraphs we give further characterizations of primitive elements and show that our definition agrees with the one in \cite{ckGAFA}.

\subsection{Primitive elements}
Recall that an infinite order element $g$ of a hyperbolic group $G$ is called {\it loxodromic}. For a loxodromic $g$, 
 \cite[Theorem 4.3]{Osin} says that $E_G^{+}(g)$ is  the maximal, virtually cyclic subgroup of $G$ of type 
finite-by-(infinite cyclic) containing $g$. It is easy to see that $E_G^{+}(g)$ is the same as $H(g)$, the subgroup of $G$ consisting of elements which fix the ends $g^-$ and $g^+$ on the boundary $\partial G$ (see Section 6 in \cite{ckGAFA}).

Since $E_G^+(g)$ acts on the infinite line $\Z$ and fixes its ends, there is a unique surjective map $\pi = \pi_g \colon E_G^+(g) \to \Z=\langle t \rangle$ with $\pi (g)= t^k, k>0$, and finite kernel $F$.

We note that our definition of primitivity coincides with that of  strong primitivity in \cite[Section 6]{ckGAFA}, which is the condition (c) below.

\begin{proposition}\label{prop:prim}
Let $G$ be hyperbolic and $g\in G$ loxodromic. The following are equivalent:
\begin{enumerate}
\item[{\rm (a)}] $g$ is primitive,
\item[{\rm (b)}] for any $h\in G$ we have $g^m \neq h^n$ for all $0<|m|<|n|$,
\item[{\rm (c)}] $\pi_g(g)=t$.
\end{enumerate}
\end{proposition}
\begin{proof}
(a)$\Rightarrow$(b). Let $F\leqslant G$ be a finite subgroup such that $E_G^+(g)=\gen{g}\cdot F$ and $E_G^+(g)\cong\gen{g}\times F$. Suppose that $g^m = h^n$ for some $0<|m|<|n|$. Then $h$ is loxodromic and $h^{-1} g^m h=g^m$, so $h \in E_G^+(g)$, which means there are some $f \in F$ and $p\neq 0$ such that $h=g^pf$, which implies $g^m=h^n=g^{np}f^n$. This furthermore implies $m-np=0$, which contradicts $0<|m|<|n|$.

$\lnot$(c)$\Rightarrow \lnot$(b).
Let $k>1$ be such that $\pi_g(g)=t^k$. There exists $h \in E_G^+(g)$ such that $\pi_g(h)=t$. Then $\pi_g(g^m)=\pi_g(h^{mk})$  for all $m>0$, and hence for every $m\in \Z$ there is  $f_m\in F=\ker (\pi_g)$ such that $g^mf_m=h^{mk}$. Since $F$ is finite, we can find $m_1$ and $m_2$, $m_1\neq m_2$ such that
$f_{m_1}=f_{m_2}$, and therefore $g^{m_1-m_2}=h^{(m_1-m_2)k}$.
Since $k>1$, $0<|m_1-m_2|<|(m_1-m_2)k|$.

(c)$\Rightarrow$(a).
From the split short exact sequence $1\to F \to E_G^+(g)\stackrel{\pi}{\to} \gen{t} \to 1,$ and $\pi(g)=t$,
we see that $\gen{g}\cdot F=E_G^+(g).$ Since both $\langle g \rangle$ and $F$ are normal in $E_G^+(g)$, and they have trivial intersection, we obtain $E_G^+(g)\cong \gen{g}\times F$.
\end{proof}

We will also need the following result, which is similar to \cite[Lemma 6.2]{ckGAFA}.
\begin{lemma}\label{lem:notprim}
If $g$ is a loxodromic element of $G$ that is not primitive, then there exist an integer $n>1$, a primitive element $g_0$ in $G$ and $r$, $|r|\leq M=M(G)$, such that $g^r=g_0^{rn}$.
\end{lemma}

\begin{proof}
Let $\tau(g)$ denote the stable translation length of $g$.
Recall that $\tau(g^n)=|n|\tau(g)$ (see \cite[III.$\Gamma$.3.14]{BridsonHaefliger}) and translations lengths
are discrete (see \cite[III.$\Gamma$.3.17]{BridsonHaefliger}) i.e. there is a positive integer $p$ such
that $p\tau(g)\in \mathbb{N}$ for all $g\in G$.

Let $g$ be a non-primitive loxodromic element and assume that the lemma holds 
for non-primitive loxodromic elements of stable translation length smaller than $\tau(g)$. 
By Proposition \ref{prop:prim}, there is $h\in G$ such that $g^m=h^n$, $0<|m|<|n|$. Note that this implies $\tau(h)<\tau(g)$.

If $h$ is primitive, then $g\in E_G^+(h)=\gen{h}\cdot F \cong \gen{h}\times F$, where $F$ is some finite subgroup of $G$.  So
$g=h^k f$, where $f \in F$ and $k\in \Z$, and hence there is $r \leq M$ such that $g^r=h^{rk}$, and the conclusion of the lemma holds.

If $h$ is not primitive, then since $\tau(h)<\tau(g)$, the induction assumption implies that $h^p=g_0^{pk}$ for some $g_0$ primitive, so $g^{mp}=h^{np}=g_0^{npk}$, and hence  $g\in E_G^+(g_0)$.
The lemma follows by repeating the same argument as above. 
\end{proof}
\begin{definition}
For $g$ and $g_0$ as in the statement of the lemma, we say that $g_0$ is a {\it root} of $g$.
\end{definition}

\subsection{The lengths of roots}
Now we make precise the intuition that a root of a non-primitive element $g$ must have
length at most $\frac{|g|_X}{2}$. We will use the BCD property of Definition \ref{def:BCD} several times.

Let $\conj\subconj$ denote the equivalence relation given
by conjugacy, $G/\!\conj\subconj$ denote its set of equivalence classes, and recall that $[g]\subconj$ denotes the conjugacy class of $g \in G$.
Define the {\em length $|g|\subconj$ up to conjugacy} of an element $g$ of $G$ by 
\[|g|\subconj:=\min\{|h|_X \mid h \in [g]\subconj\}.\]
We make the analogous definitions for $\sim\subcomm$, the commensurability relation.

Recall that in a hyperbolic space, local progress guarantees global progress 
(see for example \cite[III.H. Theorem 1.13]{BridsonHaefliger}).
That is, for a fixed $\lambda >1$ there exist $K>1,$  $\e\geq 0$, such that any $K$-local
geodesic in $\ga(G,X)$ is a $(\lambda, \e)$-quasi-geodesic. 

The following lemma, which we need later, can be proved using this local-to-global property.
\begin{lemma}\cite[Lemma 2.2.]{ckIJAC}\label{lem:ckIJAC} Let $G$ be a hyperbolic group. There is a constant $C_0 = C_0(G, X) > 0$ such that for all $g\in G$ with $|g|\subconj\geq C_0$ and for all $n\in \Z-\{0\}$, $|g^n|\subconj \geq \frac 34 |n|\cdot  |g|\subconj$.
\end{lemma}

\begin{notation} \label{not:qgconstant}
 For the rest of the section we will fix $K$ and $\e_1$ such that
every $K$-local geodesic in $\ga(G,X)$ is a $(3/2,\e_1)$-quasi-geodesic.
Let $W$ label a cyclic geodesic. Then every cyclic permutation of $W$ is a geodesic word,
so if $\ell(W)>K$ the word $W^n$ labels a $K$-local geodesic and hence a $(3/2,\e_1)$-quasi-geodesic for all $n\in \N$.
In other words, all powers of $W$ are {\em cyclic $(3/2,\e_1)$-quasi-geodesics}. 
Note that all $W \in \geocpl(G,X)$ with $\ell(W)>K$ represent loxodromic elements of $G$.
\end{notation} 

Now we give an upper bound for the length of the roots of non-primitive elements.

\begin{lemma} \label{lem:roots} 
 There are constants $K_1$,  $\lambda_2\geq 1$ and $\e_2 \geq 0$ 
 such that
for every non-primitive  conjugacy class $[g_0]\subconj$ with $|g_0|\subconj> K_1$, there is a conjugacy representative  $g\in [g_0]\subconj$ with $|g|_X=|g|\subconj$,  a primitive element $h\in G$, and $k\in G$ such that the following hold:
\begin{enumerate}
\item[{\rm (a)}] $g^r=kh^{nr}k^{-1}$ for some $r,n\in \N$, $|r|\leq M$, $n>1$,
\item[{\rm (b)}] $|k|_X\leq K_1$, 
\item[{\rm (c)}] $|h|_X\leq 3/4|g|_X+K_1$,
\item[{\rm (d)}] $p|h|_X \leq\lambda_2 |h^p|_X +\e_2$ for all $p\in \Z$.
\end{enumerate}
\end{lemma}
\begin{proof}

Let $\cS$ be the set of cyclic geodesic words of length less than $K$ that represent loxodromic elements.
Since $\cS$ is finite, there exist $\lambda_0\geq 1, \e_0 \geq 0$ such that for $W\in \cS$ and any  $n\geq 1$, $W^n$ labels a cyclic $(\lambda_0,\e_0)$-quasi-geodesic. 

Let $\lambda_2=\max\{3/2,\lambda_0\}$ and $\e_2=\max\{\e_0, \e_1\}$. We have that \begin{equation}
\label{eq:lambda2}
\text{ if $U\in \geocpl(G,X)$ is loxodromic, then
$U^n$ is a cyclic $(\lambda_2,\e_2)$-quasi-geodesic.}
\end{equation}

Let $D_2$ be the BCD constant for $(\lambda_2,\e_2)$ (see Definition \ref{def:BCD}).
Set $K_1= \max\{3D_2+\e_1 ,K \}$.

Let $g_0$ be a  non-primitive  element with $|g_0|_X=|g_0|\subconj\geq  K$, and let $W_0\in \geocl(G,X)$ be such that $W_0=_G g_0$. Then by the Notation \ref{not:qgconstant}, $W_0^p$ is a cyclic $(3/2,\e_1)$-quasi-geodesic for all $p\in \Z$ and in particular, $g_0$ is loxodromic.
By Lemma \ref{lem:notprim}, there exist $r$, $|r|\leq M$, $n\geq 1$ and $g_1$ primitive such that $g_0^r= g_1^{rn}$.
Let $V_1$ be a minimal length representative of elements in $[g_1]\subconj$. Then $V_1^{rn} \sim_G g_1^{rn}=_G W_0^r$.

The BCD property (Definition \ref{def:BCD}) implies that there are cyclic permutations $W$ and $V$ of $W_0$ and $V_1$, respectively, and $k\in \ball_X(D_2)$ such that $W^r=_G k^{-1}V^{rn}k$. Let $g=_G W$. Then $g\in [g_0]\subconj$ and  $|g|\subconj=|g|_X$. 
Let  $h=_GV$. Then (a) and (b) are satisfied, and by \eqref{eq:lambda2}, $h$ satisfies (d).  It remains to prove (c).

If $|h|_X<K\leq K_1$ there is nothing to prove. So assume that
$|h|_X>K$. Then both $W^r$ and $V^{rn}$ label $(3/2,\e_1)$-quasi-geodesics. From $W^r=_G k^{-1}V^{rn}k$ we get that 
$$rn \ell(V)=\ell(V^{rn})\leq \frac 32  |h^{rn}|_X+\e_1 \leq \frac 32 (2|k|_X+|W^r|_X)+\e_1.$$
As $|h|_X=\ell(V)$ and $|k|_X\leq D_2$, we obtain
\begin{align*}
|h|_X& \leq \frac{3}{2rn}  (2D_2+ |W^r|_X)+ \frac{\e_1}{rn} \\
& \leq \frac{3}{2rn} |W^r|_X + (3D_2+\e_1)\\
& \leq \frac{3r}{2rn} |W|_X + (3D_2+\e_1).
\end{align*} 
Since $n\geq 2$ and $\ell(W)=|W|_X= |g|_X$, the lemma follows.
\end{proof}

\subsection{The lower bound of Theorem \ref{thm:cgrate}}

\begin{corollary}\label{cor:commareprim}
There exists $K_2>0$ such that for every non-primitive $g_0\in G$ satisfying $|g_0|_X=|g_0|\subcomm$, 
we have that $|g_0|_X<K_2$.
\end{corollary}
\begin{proof}
Let $K_1$ be the constant of Lemma \ref{lem:roots}. Assume that $g_0$ is not primitive with
$|g_0|=|g_0|\subcomm >K_1$. Note that $|g_0|=|g_0|\subcomm$ implies $|g_0|=|g_0|\subconj$.

By  Lemma \ref{lem:roots}, there exist $g\sim\subconj g_0$, $|g|_X=|g_0|\subconj$, and
a primitive element  $h\in [g]\subconj \subseteq [g_0]\subcomm$ such that
$|h|_X\leq 3/4 |g_0|_X+K_1$.
Since $|g_0|_X=|g_0|\subcomm = |h|\subcomm\leq |h|_X$, we have that $|g_0|_X\leq \frac{3}{4} |g_0|_X+K_1$ and hence
$|g_0|_X\leq 4K_1$. We let $K_2=4K_1$.
\end{proof}
\begin{lemma}\label{lem:2Mprimitive}
%Let $K_2$ be the constant of Corollary \ref{cor:commareprim}.
For every $g\in G$ loxodromic, there are at most $2M$ primitive conjugacy classes in $[g]\subcomm$.
\end{lemma}
\begin{proof}
Let $h\in [g]\subcomm$ and assume that $|h|_X=|h|\subcomm$.
Suppose that $h_1,\dots, h_p\in [g]\subcomm$ are primitive elements and suppose that $\{h,h_1,\dots, h_p\}$ are pairwise non-conjugate and $p\geq 2M$.
Then there are $n_i,m_i\in \Z$ and $k_i\in G$ such that
$h^{n_i}=k_i^{-1} h_i^{m_i}k_i$ for $i=1,2,\dots, p$. 
In particular 
$$(k_i^{-1}h_ik_i)^{-1}h^{n_i}(k_i^{-1}h_ik_i) = k_i^{-1}h_i^{m_i}k_i=h^{n_i}$$
and therefore, $k_i^{-1}h_ik_i\in E_G^+(h)$ for $i=1,\dots, p$.
Since all are primitive elements, by Proposition \ref{prop:prim}~(b),
$|n_i|=|m_i|$ and hence 
$\pi_h(k_ih_ik_i^{-1})\in  t^{\pm 1}$ for $i=1,2,\dots, p$.
Since $|\pi^{-1}_h(t^{\pm{1}})|\leq 2M$, we have
that there are $i,j$ such that
$k_ih_ik_i^{-1}= k_jh_jk_j^{-1}$, contradicting the fact that $\{h,h_1,\dots, h_p\}$ are pairwise non-conjugate.
\end{proof}

\begin{proof}[Proof of Theorem \ref{thm:cgrate} (lower bound)] 
By Corollary \ref{cor:commareprim}, if $|g|\subcomm>K_2$, then all minimal length representatives of $[g]\subcomm$ are primitive; by Lemma \ref{lem:2Mprimitive}, there are at most $2M$ primitive elements in a commensurating class of a loxodromic element. Thus $$\dfrac{\phi\subpconj(n)}{2M}\leq \phi{\subcomm}(n)\leq \phi\subpconj(n)$$
 for all $n$ greater than  $K_2$, the constant of Corollary \ref{cor:commareprim}. 
 Combining this with Theorem \ref{prim_bounds}, and observing that $\phi\subpconj(n)\leq \phi\subconj(n)$, we obtain the lower bounds of Theorem \ref{thm:cgrate}.
\end{proof}

\subsection{The upper bound of Theorem \ref{thm:cgrate}}
Since $G$ is a non-elementary hyperbolic group, there are positive constants $A_0$ and $B_0$ (see \cite{c93}) such that for all $n\geq n_0$
\begin{equation}\label{eq:growthbounds}
A_0e^{\grate n}\leq |\ball_X(n)| \leq B_0e^{\grate n} .
\end{equation}

\begin{lemma}\label{lem:negli}
There are $B_1, n_0>0$ such that for all $n>n_1$
$$\phi\subconj(n)-\phi\subpconj(n)\leq B_1n^2e^{\frac{3\grate n}{4}}.$$
\end{lemma}
\begin{proof}
Let $K_1$ be the constant of Lemma \ref{lem:roots}, and assume $n > K_1$.

Since $$\phi\subconj(n)-\phi\subpconj(n)= (\phi\subconj(n)-\phi\subconj(K_1))-(\phi\subpconj(n)-\phi\subpconj(K_1))+(\phi\subconj(K_1)-\phi\subpconj(K_1))$$
and $\phi\subconj(K_1)-\phi\subpconj(K_1)$ is constant, it is enough to show that 
$$(\phi\subconj(n)-\phi\subconj(K_1))-(\phi\subpconj(n)-\phi\subpconj(K_1))\leq   B_1n^2e^{\frac{3\grate n}{4}}. $$

To show this,  we will construct a map  
$$\alpha_n\colon \big\{[g]\subconj\in G/\sim\subconj \;\mid \; g \text{ non-primitive and } K_1\leq |g|\subconj\leq n  \big\} \to \ball_X(3n/4+K_1)$$
and  a constant $\beta$ so that $\alpha_n$ is (at most $\beta n^2$)-to-one, that is, $|\alpha_n(h)^{-1}|\leq \beta n^2$ for every $h\in \text{Im}(\alpha_n)\subseteq \ball_X(3n/4+K_1)$.

Let $g_0\in G$, $K_1\leq |g_0|\subconj \leq n$. 
By Lemma \ref{lem:roots}, there are $g\in [g_0]\subconj$,
and $h,k\in G$ such that $|g|_X=|g_0|\subconj$ and (a), (b), (c) and (d) 
of Lemma \ref{lem:roots} hold. Define $\alpha_n([g]\subconj)=h$. 

We need to bound $|\alpha_n^{-1}(h)|$.
For each conjugacy class $[g]\subconj$ meeting   $\alpha_{n}^{-1}(h)$ we can assume that the representative $g\in [g]\subconj$ satisfies that $K_1 \leq |g|_X \leq n,$ and for some $r<M$, $g^r$ is conjugate to $h^{pr}$
by an element of length at most $K_1$. Note that in this case $|h^{pr}|_X\leq |g^r|_X + 2K_1\leq  n M+2K_1$.
Therefore 
$$|\alpha_n^{-1}(h)|\leq |\ball_X(K_1)|\cdot |\{ m \in \Z \;|\; |h^m|_X \leq n M+ 2K_1\}|,$$
as the first term in the product bounds the number of possible conjugators and the second one the number of possible powers of $h$.
Note that by Lemma \ref{lem:notprim}(d), and since $|g|_X\geq K_1$, there
exist $\lambda_2\geq 1$ and $\e_2\geq 0$ so that $|m|\leq |g|_X\leq \lambda_2 |h^m|_X +\e_2$.
Thus
\begin{align*}
|\alpha_n^{-1}(h)|&\leq |\ball_X(K_1)|\cdot |\{ m \in \Z \;|\; \frac{|m|}{\lambda_2 n}-c_2  \leq n M+ 2K_1\}|\\
&\leq |\ball_X(K_1)|\cdot|\{ m \in \Z \;|\; |m| \leq \lambda_2 n(n M+ 2K_1+c_2)\}|.
\end{align*}

Now by letting $\beta = 2|\ball_X(K_1)| \cdot \lambda_2 (M+2K_1+c_2)$, we obtain that the map $\alpha_n$ is an (at most $\beta n^2$)-to-$1$ map.
Let $B_1=\beta B_0 e^{\grate K_1}$, where $B_0$ is the constant of \eqref{eq:growthbounds}. 
The lemma then immediately follows since $|\ball_X(3/4n+K_1)|\leq B_0 e^{3\grate n/4} e^{\grate K_1}$.
\end{proof}

The next proof follows the same strategy as \cite{ckIJAC}, where the torsion-free case is proved.

\begin{proof}[Proof of the upper bound of Theorem \ref{thm:cgrate}]

Let $M=M(G)$ be as in \eqref{eq:finite}, let $W\equiv x_1x_2\cdots x_n$ be a cyclic geodesic representing a primitive element, and for $i=1,2,\dots, n$ consider the cyclic permutations $W_i\equiv x_i \cdots x_nx_1\cdots x_{i-1}$ of $W$.
For a fixed $i\in \{1,\dots, n\}$ we are going to show that $|\{j\mid W_j=_G W_i\}|\leq M$.

Suppose that $W_i=_G W_j$, 
$i\neq j$.
Let $U\equiv x_i\cdots x_{j-1}$ and $V\equiv x_{j} \cdots x_{i-1}$,
where the indices are considered cyclically. 
Then $UV\equiv W_i$ and $VU\equiv W_j$. 
Let $g=_G W_i$, $h=_G U$ and $k=_G V$.
Then $h,k\in C_G(g)$, the centralizer of $g$ in $G$. 
%By Lemma \ref{lem:primitive}, $g$ is primitive and since $W\in \cL\subconj$, $|g|\subconj =|g|$. 
Since $g$ is primitive and $C_G(g)\leqslant E_G^+(g)$, we can assume that $\pi_g(g)=t$ and  $C_G(g)\cong\gen{g}\times F$, where $F$ is some finite subgroup.
Then  $h=g^pf_1$, $k=g^qf_2$, with $f_i\in F$ and $p,q\in \Z$.
Since $hk=g$, $\pi_g(hk)=t^{p+q}$, and thus $p+q=1$.

By Lemma \ref{lem:ckIJAC}, there is a constant $K_0$ such that
for every $a\in G$, if $|a|\subconj\geq K_0$ then $|a^n|\subconj\geq  \frac 34 |n|\cdot |a|\subconj$. Since $h=g^pf_1$, we have that
for some $s$, $h^s=g^{ps}$.
Thus, if $\ell(W)=|g|=|g|\subconj\geq K_0$,  $|h^s|\subconj=|g^{ps}|\subconj \geq \frac 34 |p| \cdot |s| \cdot  \ell(W).$
Since $|h|<\ell(W)$, we have that $|h^s|\subconj\leq |s| |h|< |s| \cdot \ell(W)$, so 
$$\frac 34  |p| \cdot |s| \cdot  \ell(W) < |s| \cdot \ell(W), $$
and therefore $|p|\leq 1$.
A similar argument gives $|q|\leq 1$.

Thus $\max\{|p|,|q|\}=1$ and $p+q=1$, so either $p$ or $q$ is equal to $0$,
and $h$ or $k$ belong to $F$. Recall that $i$ is fixed and $W$ is a cyclic geodesic,
hence for every value of $j$ we get a different group element. Since $|F|\leq M$, we have $|\{j\mid W_j=_G W_i\}|\leq M$ as desired.

Hence, we have proved that for all primitive element $g$ with $|g|\subpconj = |g|_X\geq K_0$, there exist
at least $\lfloor\frac{|g|_X}{M}\rfloor$ distinct elements
in the conjugacy class $[g]\subconj$. Let $\sigma\subpconj(n)$ denote the number of primitive conjugacy classes of length exactly $n$.
We have that 
$$\left\lfloor\frac{n}{M}\right\rfloor \sigma\subpconj(n) \leq B_0 e^{\grate n},$$ 
where $B_0$ is the constant of \eqref{eq:growthbounds}.
Therefore, we have shown that there are $n_1$ and $D$ such that
for all $n\geq n_1$, $\sigma_{\cL\subpconj}(n) \leq D \frac{e^{\grate n}}{n}$.
By a standard argument (see  \cite[Lemma 3.2]{ckIJAC}), it follows that $\phi_{\cL\subpconj}(n)=\sum_{i=0}^n \sigma_{\cL\subpconj}(i)$ satisfies that
there are $B'>0$ and $n_0'$ such that for all $n\geq n_0'$
$$\phi\subpconj(n) \leq B' \frac{e^{\grate n}}{n}.$$ 
By Lemma \ref{lem:negli} for $n\geq \max\{n_1,n_0'\}$
$$\phi\subconj(n)= \phi\subpconj(n)+\big(\phi\subconj(n)-\phi\subpconj(n)\big)\leq B' \frac{e^{\grate n}}{n}+ B_1n^2e^{\frac{3\grate n}{4}}.$$
Since $G$ is non-elementary, $\grate >1$, and there is an $n_2>0$ such that for $n>n_2$, $\frac{1}{e^{\grate n/4}} \leq \frac{1}{n^3}$, so $B_1 e^{3\grate n/4}\leq B_1 \dfrac{e^{\grate n}}{n}$. The upper bound follows for $B=B_1+B'$.
\end{proof}

%%%%%%%%%%%%%%%%%%%%%%%%%%%%%%
%%%%% Non regularity of the subgroupp
%%%%%%%%%%%%%%%%%%%%%%%%%%%%

\section{Determining the language complexity of \texorpdfstring{$\cL\subconj$}{Lconj}, \texorpdfstring{$\cL\subpconj$}{Lprim} and \texorpdfstring{$\cL\subcomm$}{Lcomm}}\label{sec:languages}

In this section we  give sufficient conditions which, when satisfied by a group, imply that
a language of minimal length conjugacy/primitive conjugacy/commensurability representatives cannot be unambiguous context-free, which in turn implies non-regular.

\subsection{Languages and operations on languages}

All alphabets considered here are finite. We start by defining a finite state automaton, which informally can be viewed as a finite labelled graph with edge labels from a given alphabet, where the vertices are called \emph{states}, and among them one is distinguished as the \emph{start state}.
\begin{definition}
 
A {\it finite state automaton} over $A$ is a quintuple $(Q, A, \delta, q_0, F)$, where $Q$ is the finite set of \textit{states}, $\delta \colon Q\times A\to Q$ the \textit{transition function}, $q_0$ the \textit{initial} or \textit{start} state, and $F \subseteq Q$ the set of \textit{final} of \textit{accepting} states.
A word $W\equiv a_1\dots a_n$ is accepted by the finite state automaton if
$\delta(\dots (\delta(\delta(q_0,a_1),a_2)\dots,a_n)\in F$.

A language over  $A$ is {\it regular} if it consists of the words accepted by a finite state automaton.
\end{definition}

The proofs of the statements  (1)--(3) in the following lemma can be found in \cite{hu}.% and the statement (4) is \cite[Lemma 2.1]{CHHR}.
\begin{lemma}\label{lem:regularfacts}
Let $A$ and $B$ be two alphabets and $\cL$ and $\cM$ be regular languages over $A$.
\begin{enumerate}
\item[{\rm (1)}] $\cL\cup \cM$, $\cL \cap  \cM$ and $\cM-\cL$ are regular,
\item[{\rm (2)}] if $\phi\colon A^*\to B^*$ is a monoid morphism, then $\phi(\cL)$ is regular,
\item[{\rm (3)}] if $\phi \colon B^*\to A^*$ is a monoid morphism, then $\phi^{-1}(\cL)$ is regular.
%\item[{\rm (4)}] $\cyc(\cL)=\{UV \mid VU\in \cL\}$ is regular.
\end{enumerate}
\end{lemma}

A \emph{context-free language} is recognized by a finite state automaton with additional memory, or by a so-called \emph{context-free grammar}. If this grammar produces each word in a unique manner we say that the language is \emph{unambiguous context-free}.
Recall that by the Chomsky-Sch\"utzemberger Theorem and the discussion in Section \ref{sec:series}, if $\cL$ is unambiguous context-free, then
$\sum_{W\in \cL} z^{\ell(W)}$ is algebraic. In \cite{Flajolet} Flajolet gave many examples of context-free languages with transcendental growth series, therefore the adjective `unambiguous' is necessary in the statement of Theorem \ref{thm:AHregular}. The following corollary and remark record the fact, applied to the languages studied in this paper, that if the growth series of a languages is transcendental, then the language cannot be unambiguous context-free.

\begin{corollary}\label{cor:notUCF}
Let $H$ be a non-elementary hyperbolic group with finite symmetric generating set $Y$.
Let $\cL \subseteq Y^*$ be a language satisfying one of the following properties:
\begin{enumerate}
\item[{\rm (1)}] $\cL$ contains exactly one minimal length representative of each $H$-conjugacy class.
\item[{\rm (2)}] $\cL$ contains exactly one minimal length representative of each $H$-conjugacy class of primitive elements.
\item[{\rm (3)}] $\cL$ contains exactly one minimal length representative of each $H$-commensurability class of primitive elements.
\item[{\rm (4)}] $\cL$ contains exactly one minimal length primitive element representative in each $H$-commensurability class.
\end{enumerate}
Then $\cL$ is not unambiguous context-free.
\end{corollary}
\begin{proof}
In each case $\sum_{n\geq 0}\phi_\cL(n)$ is transcendental: (1), (2) and (3) follow from Theorem \ref{thm:rivinconj} and (4) follows from Remark \ref{rem:quasicommgrowth} in view of Lemma \ref{lem:2Mprimitive}. Then, by the Chomsky-Sch\"utzemberger Theorem, $\cL$ is not unambiguous context-free.
\end{proof}

The only facts about unambiguous context-free languages used in the proofs are collected in the following lemma. The first statement follows from  \cite{GinsburgUllian1}. Statements (2) and (3) follow from \cite{GinsburgUllian2}, recalling that monoid morphisms are
special cases of gsm (generalized sequential machine) mappings. The last statement follows from (1). 
\begin{lemma}\label{lem:CFfacts}
Let $\cL$ be an unambiguous context-free language over $A$. 
\begin{enumerate}
\item[{\rm (1)}] If $\cM$ is a regular language over $A$, then $\cL\cap \cM$ and $\cL \cup \cM$ are unambiguous context-free.
\item[{\rm (2)}] If $\phi\colon A^*\to B^*$ is a monoid morphism and $\phi$ restricted to $\cL$ is injective, then $\phi(\cL)$ is 
unambiguous context-free.
\item[{\rm (3)}]  If $\phi\colon  B^*\to A^*$ is a monoid morphism, then $\phi^{-1}(\cL)$ is unambiguous context-free.
\item[{\rm (4)}] If $\cM$ is a language over $A$ that has finite symmetric difference with $\cL$, 
then $\cM$ is unambiguous context-free.
\end{enumerate}
\end{lemma}

The strategy we use is as follows: suppose that $G$ is finitely generated by $X$, $\cL_c$ 
is a language of conjugacy/primitive conjugacy/commensurability representatives, and there exists a non-elementary hyperbolic subgroup $H$ of
$G$  `nicely' embedded in $G$. We will construct a finite generating set $Y$ of $H$ and a finite state automaton (Lemmas \ref{lem:BCDpairs} -- \ref{lem:defDelta}) which produces a language $\cR$ of $H$-conjugacy representatives over $Y$ related to $\cL_c$ in such a way that if $\cL_c$ is unambiguous context-free, then so is $\cR$. But then we get a contradiction since $\cR$ cannot be unambiguous context-free by Corollary \ref{cor:notUCF}.

In order to follow this approach, the automaton will need to perform transitions 
on words over $X$ and words over $Y$ simultaneously. This fact requires an additional layer of technicality.

\begin{definition}\label{def:padded}
Let $A_1,A_2, 
\dots, A_k$ be  alphabets. 
For each $A_i$, let $\$_i$ be a {\it padding symbol} which we assume does not lie in $A_i$. 
Let $B_i=A_i \cup \{\$_i\}$.
We  let $\unpad\colon B_i^*\to A_i^*$ denote the monoid morphism which deletes the $\$_i$-symbol.
The {\it padded alphabet} associated with $(A_1, A_2, 
\dots, A_k)$ is the set $B= B_1\times \dots 
\times B_k.$
For $1 \leq i,j \leq k$ we use $\pi_{i}\colon B^*\to B_i$ and $\pi_{ij}\colon B^*\to (B_i\times B_j)^*$ to denote projection morphisms on the respective coordinates.
We remark that we differ slightly from \cite[Section 1.4]{ECHLPT} since we do not require the removal of $\{(\$_1,\dots,\$_k)\}$ from $B$.
\end{definition}

\subsection{Determining the language complexity: a geometric criterium}
For the rest of the section we will use the following notation.
\begin{notation}\label{not:lang}
Let $G$ be a group with a finite symmetric generating set $Z$. 
Suppose that $X$ and $Y$  are finite sets with  given maps $X\to G$ and $Y\to G$, 
so that we view the elements of $X$ and $Y$ as elements of $G$. 
The padding symbols $\$_X$ and $\$_Y$ are identified with the trivial element of $G$.
We let $$X^{\$}:= X\cup \{\$_X\}, Y^{\$}:=Y\cup \{\$_Y\}  \ \textrm{and} \ B:= X^{\$}\times Y^{\$},$$ and fix a total order $<_{\mathsf{lex}}$ on $Y^{\$}$ which we extend to the usual lexicographic order on words over $Y^{\$}$. Typically $Y$ will refer to the generating set of a hyperbolic (sub)group.
%, and furthermore to pairs  of words of the same length over $Y^{\$}$.
\end{notation}
\begin{notation}
Suppose that $\gen{X}=G$ and $\gen{Y}=H\leqslant G$.
For convenience we will write
$$\geocl(G,X)\cap H^{G}$$
instead of 
$$\geocl(G,X)\cap \{U\in X^*\mid U=_G h\in H^{G}\}.$$
In general, when $\cL\subset X^*$ is a language and $A\subset \gen{X}$ is a subset of a group, we use $\cL\cap A$ to denote the set of words of $\cL$ that represents elements in $A$. 
\end{notation}

\begin{definition}\label{def:fellowtrav} 
Let $(U,V)\in B^*$ and let $U_j$ (resp $V_j$) denote the $j$-th prefix of $U$ (resp. $V$). If $j>\ell(U)$, then $U_j\equiv U$, and similarly for $V$.
Let $g\in G$. We say that {\it $gU$ and $V$ synchronously $K$-fellow travel} in $(G, \d_Z)$
$$\d_Z(gU_{j}, V_{j})\leq K \qquad \text{ or, equivalently, }\qquad 
|V_{j}^{-1}gU_{j}|_Z\leq K \qquad \text{ for } j=0,1,2,\dots.$$
\end{definition}

\begin{definition}\label{def:FTC}
The pair $(U,V) \in B^*$ is a {\it $K$-synchronous BCD pair} in $(G,\d_Z)$ if 
 \begin{enumerate}
\item[(1)] there exists $g \in \ball_Z(K)$ such that  $gU=_G  V g$,

\item[(2)] $V$ synchronously $K$-fellow travels with $gU$ in $(G,\d_Z)$.
\end{enumerate}
\end{definition}
Recall that for a language $\cL$ we denote by $\cyc(\cL)$ the set of all cyclic shifts of words in $\cL$. That is, $\cyc(\cL)=\{UV \mid VU\in \cL\}$.

\begin{remark}\label{rem:BCDpair_inv_cyc}
 \hspace{-1cm} \begin{enumerate}
\item[{\rm (1)}] Our definition of synchronous fellow travelers is for words; geometrically, when these words are seen as paths in the Cayley graph, the fellow traveling might be  asynchronous, since each $\$$ letter can be interpreted as not moving along the path.

\item[{\rm (2)}] Notice that if $(U,V)$ is a $K$-synchronous BCD pair in $(G,\d_Z)$, then so is any cyclic permutation $(U',V')\in \cyc (\{(U,V)\})$. 
\end{enumerate}
\end{remark}

All distances are assumed to be in $(G,\d_Z)$, unless otherwise specified. We will say that 
$U$ and $V$ synchronously $K$-fellow travel or $(U,V)$ is a $K$-synchronous BCD pair, and not mention the ambient $(G,\d_Z)$.

The next lemmas are standard results on regular languages, and we only sketch the proofs.

\begin{lemma}\label{lem:BCDpairs}
Let $K\geq 0$. The following set is a regular language:
 $$\cM=\{(U,V)\in B^* \mid \text{  $(U,V)$ is a $K$-synchronous BCD pair} \}.$$
\end{lemma}
\begin{proof}[Sketch of the proof]
Let $g\in \ball_Z(K)$.
We claim that the following set is a regular language:
$$\cL_g = \{(U,V)\in B^* \;\vline \ gU  =_G Vg,  \text{  $(U,V)$ is a $K$-synchronous BCD pair} \}.$$

Indeed, we can construct a finite state automaton with alphabet $B$ and states  $\ball_Z(K)\cup \{\rho\}$, 
where $\{g\}$ is the initial and only accepting state.
The state $\rho$ is a fail state, that is, $\tau(\rho, b)=\rho$ for all $b\in B$.
Let $b=(x,y)\in B$ and $h\in \ball_X(K)$. The transition function $\tau$ on $(h,b)$ is given by 
$\tau(h, (x,y))= x^{-1}hy$ if $x^{-1}hy\in \ball_X(K)$ and  $\tau(h, (x,y))=\rho$ otherwise.

It is an easy exercise to check that this automaton accepts exactly the language $\cL_g$.

Since regular languages are closed under finite unions,
$\cM= \bigcup_{g\in \ball_Z(K)}  \cL_g, $
is regular.
\end{proof}

\begin{lemma}\label{lem:lex}
The set 
$\cM_1=\{(V_1,V_2)\in (Y^{\$}\times Y^{\$})^* \mid V_1<_{\mathsf{lex}}V_2\}$
is a regular language.
\end{lemma}
\begin{proof}[Sketch of the proof]
Construct a finite state automaton with states $\{-1,0,1\}$, where $-1$ is the accepting state, $0$ is the initial state, and the transition function is defined by $T(1,(a,b))=1, T(-1,(a,b))=-1$ for all $(a,b)\in Y^{\$}\times Y^{\$}$, $T(0,(a,b))= -1$ if $a<_\mathsf{lex} b$, $T(0,(a,b))=0$ if $a=b$ and $T(0,(a,b))=1$ if $b<_{\mathsf{lex}}a$.
\end{proof}

The following lemma can be seen as associating, via regular operations, to each $U \in X^*$ that is conjugate to some element in a given subset of $\gen{Y}$, a lexicographically minimal word $V \in Y^*$ in $U$'s conjugacy class.

\begin{lemma}\label{lem:defDelta}
Suppose that $\cS\subseteq (Y^{\$})^*$ is a regular language.
%Suppose that $\cS\subseteq B_2^*$ is a regular language.
The language
$$\cM_2=\{(U,V)\in B^* \mid V\equiv \min_{\leq_{\mathsf{lex}}} ( V'\in \cS \mid (U,V')\text{ is a $K$-synchronous BCD pair} )\}$$
is regular.
\end{lemma}
\begin{proof}[Sketch of the proof]
Let $C= X^{\$}\times Y^{\$}\times Y^{\$}$. As in Definition \ref{def:padded} we use $\pi_{ij}$ and $\pi_i$ to denote projections from $C^*$ to the appropriate coordinates. 
Let $\cM$ be the regular language of Lemma \ref{lem:BCDpairs}. 
Then $$\cT_1=\{(U,V_1,V_2) \in C^{*} \mid (U,V_1), (U,V_2) \text{ are $K$-synchronous BCD pairs, }V_1,V_2\in \cS\}$$
is regular since it is the intersection of the pre-images of several regular languages: $\cT_1= \pi_{12}^{-1}(\cM) \cap \pi_{13}^{-1}(\cM) \cap \pi_2^{-1}(\cS) \cap \pi_3^{-1}(\cS).$

The language 
$$\cT_2=\{(U,V_1,V_2) \in \cT_1 \mid V_1<_{\mathsf{lex}}V_2 \}$$
is regular, since it is the intersection of $\cT_1$ with the pre-image of the regular language $\cM_1$ of Lemma \ref{lem:lex} under the map $\pi_{23}$.
Since $\cM_2=\pi_{12}(\cT_2)-\pi_{13}(\cT_2)$, the result follows from the closure properties of regular languages in Lemma \ref{lem:regularfacts}.
\end{proof}

\begin{lemma}\label{lem:BCD->syncBCD}
 Assume that $X=Y=Z$ in the Notation \ref{not:lang} and $\ga(G,Y)$ is $\delta$-hyperbolic.
 
Let $\lambda\geq 1$ and $c\geq 0$.
Then there is a constant $K=K(\delta, \lambda, c)$ such that 
for every cyclic $(\lambda,c)$-quasi-geodesic $U\in Y^*$ we can find
$V\in (Y^{\$})^*$ satisfying that $\unpad(V)\in \geocl(G,Y)$ and $(U,V)$ is a $K$-synchronous BCD pair.
\end{lemma}
\begin{proof}
Let $D=D(\lambda,c,\delta)$ be the BCD constant.
Let $U$ be a cyclic $(\lambda,c)$-quasi-geodesic word with $\ell(U)\leq D$ and let $V'\in \geocl(G,Y)$ be conjugate to $U$. Since $\ell(V')\leq \ell(U)$, we can pad $V'$ with symbols $\$_Y$ at the end to get a word $V$ so that $\ell(U)=\ell(V)$ and $\unpad(V)\equiv V'$. There exists a number $K_U$ so that $(U,V)$ is a $K_U$-synchronous BCD pair.
Let $K_0$ be the maximum of all $K_U$ where $U$ be a cyclic $(\lambda,c)$-quasi-geodesic words $\ell(U)\leq D$.

Thus we only need to consider the case where $U$ is a cyclic $(\lambda,c)$-quasi-geodesic of length $\ell(U)>D$.
In this case take any $\ol{V}\in\geocl(G,Y)\cap U^G$. By the BCD property
there are cyclic permutations $\ol{V}'$ of $\ol{V}$ and $U'$ of $U$, and
$g\in \ball_Y(D)$ such that $gU'=\ol{V}'g$. 
Let $p$ be the path in $\ga(G,Y)$ starting at $1$ and labelled by $\ol{V}'$, and
$q$ the path starting at $g$ and labelled by $U'$. Then $p$ is a geodesic path,
$q$ is a $(\lambda,c)$-quasi-geodesic path, $\d_Y(p_-,q_-)\leq D$ and
$\d_Y(p_+,q_+)\leq D$.

For $i\in \{0, 1,\dots, \ell(\ol{V}')\}$ let $t_i\in \{0,1,\dots, \ell(U)\}$ be
the biggest number such that $\d_Y(1, q(t_i))=i$. By the stability of quasi-geodesics there is a constant
$\kappa=\kappa(\delta,\lambda,c,D)$ such that every vertex of $q$ is at distance at most
$\kappa$ of a vertex of $p$, so in particular
there exists a point $v$ on $p$ such that $d(q(t_i), v) \leq  \kappa$. Then $$\d(1,p(i))-\kappa  = \d(1, q(t_i))-\kappa \leq \d(1,v)\leq \d(1,q(t_i))+\kappa  =\d(1, p(i))+\kappa.$$ 
Since $p$ is geodesic, we get $d(v, p(i)) \leq \kappa$, and by the triangle inequality $\d_Y(p(i),q(t_i))\leq 2\kappa$.

Since $q$ is a $(\lambda,c)$-quasi-geodesic and $\d_Y(q(t_i),q(t_{i+1}))\leq 2\kappa +1$,
we see that  $$t_{i+1}-t_i\leq \lambda (2\kappa+1)+c.$$

We now construct $V'$ by padding $\ol{V}'$ as follows:
for $j\in \{1,\dots, \ell(U)\}$ the $j$th letter of $V'$ is the $i$th letter of $\ol{V}'$ if $j=t_i$ and
$\$_Y$ otherwise. Note that $\d_Y(U'_i,V'_i)\leq 2\kappa + \lambda (2\kappa+1)+c$.

Then $(U',V')$ is a $(2\kappa+\lambda(2\kappa+1)+c)$-synchronous BCD pair.
By Remark \ref{rem:BCDpair_inv_cyc}, there is a cyclic permutation $V$ of $V'$
such that $(U,V)$ is a $(2\kappa+\lambda(2\kappa+1)+c)$-synchronous BCD pair.
Note that since $\geocl(G,Y)$ is closed under cyclic permutation,
$\unpad(V)\in \geocl(G,Y)$.

The lemma now follows with $K=\max\{K_0,(2\kappa+\lambda(2\kappa+1)+c) \}$.
\end{proof}

A subset $H$ of $\ga(G,X)$ is {\it Morse} if
for any $\lambda \geq 1, c\geq 0$ there exists $\kappa=\kappa(\lambda, c)$ such that every $(\lambda,c)$-quasi-geodesic in $\ga(G,X)$ with end points  
in $H$ lies in the $\kappa$-neighborhood of $H$.
Note that being Morse is stable under quasi-isometry. Note also that if a finitely generated subgroup $H$ of $G$ is a Morse subspace of $\ga(G,X)$, then $H$ is quasi-isometrically embedded in $(G,X)$.

\begin{definition}\label{def:BCD-embedded}
Let $H$ be a subgroup of $G$ and $X$ a finite generating set of $G$. We say that $(G,X)$ has {\it BCD relative to $H$} if there is $K\geq 0$ such that for every $U\in \geocl(G,X)\cap H^G$ we can find $g\in \ball_X(K)$ and a cyclic permutation $U'$ of $U$ so that $U' \in H^g$.
\end{definition}

We will combine the Morse and relative BCD properties of $H$ and $G$ to build BCD pairs.

\begin{proposition}\label{prop:YandR}
Suppose that $G$ is generated by a finite symmetric set $X$ and $H\leqslant G$ is hyperbolic and  Morse. Also suppose that $(G,X)$ has BCD relative to $H$.

There exists a finite symmetric generating set $Y$ of $H$ and a constant $R=R(K,X,Y,G)$ for which to any
$U\in \geocl(G,X)\cap H^G$ we can associate $V\in (Y^{\$})^*$ such that
$\unpad(V)\in \geocl(H,Y)$ and 
$(U,V)\in B^*$ is an $R$-synchronous BCD pair. 
\end{proposition}
\begin{proof}
We use Notation \ref{not:lang}, with the assumption that $X=Z$; the set $Y$ will be defined below.

Let $K$ be the constant provided by the relative BCD property of $(G,X)$ with respect to $H$.

Let $\kappa=\kappa(1,K)$ be the 
Morse constant for $H$, that is, any $(1,K)$-quasi-geodesic in $\ga(G,X)$ with end points in $H$ lies in the 
$\kappa$-neighbourhood of $H$. % Note that $K\leq \kappa$.
 We choose a finite symmetric generating set $Y$ (that may include the trivial element) of $H$ such that 
\begin{equation}\label{eq:Y contains 2kappa+1}
Y \supseteq \{h\in H \mid |h|_X\leq 2\kappa +1\}.
\end{equation}

Since $H$ is quasi-isometrically embedded in $G$, there exists $D>0$ such that
\begin{equation}\label{eq: D}
\text{for all $h\in H$, $\frac{1}{D}|h|_X-D\leq |h|_{Y}\leq D |h|_X+D$.}
\end{equation}

Let $U\in \geocl(G,X)\cap H^G$ and let $K$ be the constant of the BCD-embedding of $H$. Then  there are a cyclic permutation $U'$ of $U$ and $g\in \ball_X(K)$ such that $gU'g^{-1} \in H$. 
Let $p$ be the path in $\ga(G,X)$ with label $U'$ starting at $g$.
Then $p$ is a geodesic path that has end points at distance  $|g|_X\leq K$ from $H$, and hence $p$ is a subpath of   a $(1, K)$-quasi-geodesic with end points in $H$; therefore, since $H$ is Morse, $p$ lies in the $\kappa$-neighbourhood of $H$.
In particular, for each vertex $p(t)$ of $p$, 
 there is a vertex $u_t$ in $H$ such that
$\d_X(p(t),u_t)\leq \kappa$.
We note that $\d_X(u_t, u_{t+1})\leq 2 \kappa +1$ and
then, by \eqref{eq:Y contains 2kappa+1}, $\d_Y(u_t,u_{t+1})\leq 1$.
Let $W$ be the word over $Y$ whose $i$th letter is $u_{i}u_{i-1}^{-1}$ for $i \geq 1$.
Insisting that we may have the trivial element as a generator in $Y$,
we have that $(U',W)\in B^*$ is a $\kappa$-synchronous  BCD pair.
In view of Remark \ref{rem:BCDpair_inv_cyc}, and changing $W$ by a cyclic permutation,
we can assume that $(U,W)\in B^*$ is a $\kappa$-synchronous  BCD pair.

Observe that $W$ is a cyclic $(D,D^2+2\kappa)$-quasi-geodesic in $\ga(H,Y)$.
Indeed, let $1\leq i \leq j \leq \ell(W)$ and $W_0\equiv (u_{i+1}u_{i}^{-1}) \cdots (u_ju_{j-1}^{-1})$ be a subword of $W$. Then $\ell(W_0)= j-i$. Note that 
$\d_X(p(i),p(j))=j-i \leq |W_0|_X +2\kappa$ and then by \eqref{eq: D}
\begin{align*}
\ell(W_0)  = |i-j| & \leq |W_0|_X+2\kappa\\
&\leq D |W_0|_Y +D^2+2\kappa
 \end{align*}
This shows that $W$ is a $(D,D^2+2\kappa)$-quasi-geodesic. To see that it is in fact a cyclic $(D,D^2+2\kappa)$-quasi-geodesic, observe that if $(U',W')$ is a cyclic permutation of $(U,V)$, then $(U',W')$ is a  $\kappa$-synchronous BCD pair with $U'$ geodesic, so the same argument as above applies to $W'$. 

Since $W$ is a cyclic $(D,D^2+2\kappa)$-quasi-geodesic and $\ga(H,Y)$ is hyperbolic,
by Lemma \ref{lem:BCD->syncBCD} there is a constant $K_2$ only depending on $D$, $\kappa$ and the 
hyperbolicity constant of $\ga(H,Y)$ so that there is 
 $V\in (Y^{\$})^*$ for which $\unpad(V)\in \geocl(H,Y)$ and
$(W,V)$ is a $K_2$-synchronous BCD-pair.

Let $R=\kappa+K_2$. It follows that
$(U,V)$ is an $R$-synchronous BCD pair.
\end{proof}

Recall (see \cite{Thompson}) that a subgroup $H$ of a group $G$ is called {\it Frattini embedded} if any
two elements of $H$ that are conjugate in $G$ are also conjugate in $H$, in other words, for all $h\in H$, $h^G\cap H=h^H$.  We will say that a subgroup $H$ of $G$ is {\it almost Frattini embedded} if for all but finitely many $H$-conjugacy classes of elements of $H$, one has that $h^G\cap H=h^H$.

Recall that $H$ is {\it almost malnormal} in $G$ if for every $g\in G-H$, $|H\cap H^g|<\infty$. In particular, if $H$ is a subgroup of $G$ with finitely many $H$-conjugacy classes of finite order elements, we have that $H$ is almost Frattini embedded in $G$.

Suppose that $H$ is almost malnormal and some infinite order element $g\in G$  is $G$-commensurated to some $h\in H$, in this case $g$ has to be conjugated to some element of $H$. Indeed, if $cg^nc^{-1}=h^m$ for some $c\in G$, $n,m\neq 0$, then $H\cap H^{cgc^{-1}}$ is infinite and thus $cgc^{-1}\in H$.

\begin{theorem}\label{thm:crit}
Let $G$ be a group with finite generating set $X$. 
Let $H$ be a  Morse, non-elementary hyperbolic, almost Frattini embedded subgroup of $G$, and assume that $(G,X)$ has BCD relative to $H$.

Then any language of minimal length conjugacy/ primitive conjugacy $G$-representatives is not unambiguous context-free.
If in addition $H$ is almost malnormal, then any language of minimal length commensurability $G$-representatives is not unambiguous context-free.
\end{theorem}

\begin{proof}
 We assume $Z=X$ in the Notation \ref{not:lang}. Let $Y$ and $R$ be the finite set and the constant provided by Proposition \ref{prop:YandR}, respectively.

Let $K$ be the constant provided by the BCD-embedding.

The first step of the proof is to discard the %different $H$-conjugacy classes that can be $G$-conjugated.
set $\cF$ containing all words in $\geocl(H,Y)$ which represent elements of $H$ such that $h^G\cap H\neq h^H$. Since $H$ almost Frattini embedded, $\cF$ is finite. 
For simplicity, we will also assume that $\cF$ contains all representatives of finite order elements  in $\geocl(H,Y)$. Recall that there are only finitely many $H$-conjugacy classes of finite order elements, and hence $\cF$ is still finite.
Let $$\cS:= \{V\in (Y^{\$})^* \mid \unpad(V)\in \geocl(H,Y)-\cF\}.$$
Since $H$ is hyperbolic, by \cite[Theorem 3.1]{CHHR}  $\geocl(H,Y)$ is regular.
Since $\cF$ is finite, the language $\geocl(H,Y)-\cF$ is regular.
Now $\cS$ is regular since it is the homomorphic pre-image of a regular language.

Consider the language 
$$\cM_2= \{(U,V)\in B^* \mid V\equiv \min_{\leq_{\mathsf{lex}}} ( V'\in \cS \mid (U,V')\text{ is an $R$-synchronous BCD pair})\}$$
By Lemma \ref{lem:defDelta}, $\cM_2$ is regular.

The second step is to build a map $\Delta$ which associates to each conjugacy representative in $H^G$ a conjugacy representative in $H$ over $Y$ as follows. Let $\cL\subseteq \geocl(G,X)$ be a language containing at most one element for each $G$-conjugacy class.
By Proposition \ref{prop:YandR}, for every $U\in \cL \cap H^G$, there is
$V\in (Y^{\$})^*$ such that $(U,V)\in B^*$ is an $R$-synchronous BCD pair. Therefore, we have that
  $$\cL \cap (H-\cF^H)^G= \cL \cap \pi_1(\cM_2).$$
Furthermore, for $U\in \cL \cap (H-\cF^H)^G$
there is a unique $V\in (Y^{\$})^*$ such that $(U,V)\in \cM_2$. We define $$\Delta(U)=\unpad(V)$$ and note that since $(U,V) \in \cM_2$, $U$ is conjugate to $\Delta(U)$.

If $\cL$ is unambiguous context-free, then so is 
$\cM_0=\{(U,V)\in (X\times Y^{\$})^*\mid U \in \cL\},$
since it is the  pre-image of $\cL$ under the natural monoid morphism $(X\times Y^{\$})^*\to X^*$.
In particular, since $\cM_2$ is regular, we have that $\cM_2\cap  \cM_0$ is unambiguos context-free.
Note that $\pi_1(\cM_2 \cap \cM_0)= \cL \cap (H-\cF^H)^G$.

Now we define the set
$$\cR = \Delta (\cL \cap (H-\cF^H)^G) \subseteq Y^*.$$
From the construction, it follows that $\cR= \unpad(\pi_2(\cM_2 \cap \cM_0)).$
Note that $(\unpad\circ \pi_2)$ restricted to $(\cM_2 \cap \cM_0)$ is injective.
Indeed, if $\unpad(\pi_2((U_1,V_1)))=\unpad(\pi_2((U_2,V_2)))$, then $V_1$ and $V_2$ are $G$-conjugate,
and thus $U_1$ and $U_2$ are $G$-conjugate. Since $\cL$ contains at most one element in each 
$G$-conjugacy class, $U_1\equiv U_2$. By construction, there is a unique $V_i$ such that $(U_i,V_i)\in \cM_2 \cap \cM_0$, therefore $V_1\equiv V_2$.
It follows from Lemma \ref{lem:CFfacts}(2) that if $\cL$ is unambiguous context-free, then so is $\cR$.

%Since $H$ is almost Frattini embedded, for every $h\in H-\cF^H$  we have that $h^H= h^G\cap H$. %and $[h]\subcomm^H=\{h'\in H \mid h \text{ is $H$-commensurated to $h'$}\}$ is equal to $[h]\subcomm \cap H$.

For the final step in the proof we show that if we start with $\cL$ a  language of conjugacy/primitive conjugacy/commensurating representatives then  $\cR$ is, up to some finite set, a language of some appropriate type of representatives in a hyperbolic group, that cannot be unambiguous context-free.

(i) Assume that $\cL$ is a language of minimal $X$-length conjugacy $G$-representatives.
Then $\cR\subseteq \geocl(H,Y)$ has finite symmetric difference with a set of minimal $Y$-length 
$H$-conjugacy representatives. Indeed, $\cR$ contains a representative of each $G$-conjugacy class
of elements of $(H-\cF^H)^G$, and by construction for each $h\in H-\cF^H$ we have that $h^G\cap H=h^H$, that is each $G$-conjugacy class corresponds exactly to one $H$-conjugacy class. Then by Corollary \ref{cor:notUCF}, $\cR$ can not be unambiguous context-free, and neither can be $\cL$.

(ii) The same argument works if $\cL$ is a language of minimal $X$-length primitive conjugacy $G$-representatives, 
bearing in mind that since  $U\in \cL \cap (H-\cF^H)^G$ is conjugated to $\Delta(U)$, then $\Delta(U)$ must be primitive as well. 
Again by construction, $\cR$ has finite symmetric difference with a set of minimal $Y$-length representatives of primitive $H$-conjugacy classes.

(iii) Now let $\cL$ be a language of minimal $X$-length commensurability classes in $G$ and assume $H$ is almost malnormal.
In this case we can not guarantee that $\Delta(U)$ is of $Y$-minimal length in its $H$-commensurability class;
however, by increasing $\cF$, we can ensure that $\Delta(U)$ is primitive (see Lemma \ref{lem:acylprim} below).
By Corollary \ref{cor:notUCF}, $\cR$ can not be unambiguous context-free, and neither can be $\cL$.

\end{proof}

\begin{lemma}\label{lem:acylprim}
%Suppose that $H\leqslant G$ is a hyperbolic group, quasi-isometrically embedded in $G$, malnormal and Morse.
%Assume there is 
% $K\geq 0$ such that for every $U\in \geocl(G,X)\cap H^G$ there are a cyclic permutation $U'$ of $U$ and $g\in \ball_X(K)$ so that $U' \in H^g$.
Suppose the hypotheses of Theorem \ref{thm:crit} hold, and assume that $H$ is almost malnormal.

There exists $K_2$ such that all $U\in \geocl(G,X)\cap H^G$ satisfying $|U|\subcomm=|U|_X>K_2$
are primitive.
\end{lemma}
\begin{proof}
Let $Y$ and $R$ be the generating set and constant of Proposition \ref{prop:YandR}, respectively.
Since $(H,Y)$ is hyperbolic, there are $\lambda_1\geq 1$ and $\e_1\geq 0$ such that
for all $V\in \geocl(H,Y)$ representing an infinite order element, all powers
of $V$ represent a cyclic $(\lambda_1,\e_1)$-quasi-geodesics in $(H,Y)$. 
Let $D$ be the BCD constant for $(\lambda_1,\e_1)$.

For each $y\in Y$ let $W_y$ be a geodesic word in $\geol(G,X)$ representing $y$.
Let $\ga_0$ be the subgraph of $\ga(G,X)$ containing all the paths starting at $1$
whose label is in $\{W_y \mid y\in Y\}^*$. Note that 
$\ga_0$ contains $H$ and is quasi-isometric to it.

Since $H$ is Morse, there is $\kappa$ such that all $(1,R)$-quasi-geodesics with end-points in $H$ are
in the $\kappa$-neighbourhood of $H$. Let $\ga$ be the complete subgraph of $\ga(G,X)$ spanned by all
vertices at distance $\leq \kappa$ of a vertex in $\ga_0$.
Then $\ga$ is a connected graph, quasi-isometric to $H$, and hence hyperbolic.
 Thus there are $K_0$ and $\e_0$ such that
every local $K_0$-geodesic is $(3/2,\e_0)$-quasi-geodesic.

Let $U \in \geocl(G,X)\cap H^G$, and suppose that $|U|\subcomm=\ell(U)>K_0$ but $U$ does not represent a primitive element. Then there are a primitive $W\in \geocl(G,X)$ and $r, n>1$ such that
$W^{rn}$ is conjugate in $G$ to $U^r$. By almost malnormality $W\in H^G$. Also, $|U|\subcomm=\ell(U)$ implies $\ell(W)\geq \ell(U)>K_0$. Since cyclic permutations of $U$ and $W$ are at distance $\leq R$ from
$H$, and $U$ and $W$ label cyclic geodesics in $\ga$, all powers of $U$ and $W$ label $(3/2,\e_0)$-quasi-geodesics in $\ga$.

Up to cyclic permutations, we can assume there exist $V_1$ and $V_2$ in $(Y^{\$})^*$ such that
$(U,V_1)$ and $(W,V_2)$ are $R$-synchronous BCD pairs. 
Recall that $\unpad(V_1)$ and $\unpad(V_2)$ are in $\geocl(H,Y)$ and since they represent infinite order elements, powers of $\unpad(V_1)$ and $\unpad(V_2)$ represent
$(\lambda_1,\e_1)$-quasi-geodesics. 
Then there are cyclic permutations of $V_1$ and $V_2$ such that $V_1^r$ is conjugated to $V_2^{rn}$ by
an element of $Y$-length at most $D$.

Let $D_1= 2R+D\max\{\ell(W_y)\mid y\in Y\}$. Then, up to cyclic permutations of $U$ and $W$,
we can assume that we have in $\ga$ paths $p$ and $q$ labelled by $U^r$ and $W^{rn}$ with $\d_\ga(p_-, q_-)\leq D_1$ and $\d_\ga(p_+,q_+)\leq D_1$. Arguing as in the proof of Lemma \ref{lem:roots}, we conclude that
$$\ell(W)\leq \frac 34 \ell(U)+ 3D_1+\e_0.$$
Since $\ell(U)\leq \ell(W)$ we get $\ell(U)\leq 12D_1+\e_0$.
The lemma follows with $K_2= 12D_1+\e_0$.
\end{proof}

Theorem \ref{thm:crit} can be applied, for example, when $G$ is non-elementary 
and hyperbolic with respect to some family of proper subgroups $\{H_\lambda\}_{\lambda \in \Lambda}$.
 In \cite[Theorem 4.11]{Bigdely} it is shown how to construct a non-elementary virtually free subgroup
$H$ of any non-elementary relatively hyperbolic group $G$ such that all 
elements of $H$ are loxodromic (in the relatively hyperbolic sense), and $H$ can be added to the family of parabolic
subgroups. It follows from standard facts about relatively hyperbolic groups that
such $H$  is quasi-isometrically embedded and almost malnormal. The Morse property for $H$
follows from \cite[Theorem 1.12(1)]{DS}, and the relative BCD property from \cite[Theorem 9.13]{AC3}.

%%%%%%%%%%%%%%%%%%%%%%%%%%%%%%
%%%%%%%%%%%%%%%%%%%%%%%%%%%%%%%
%%%%%%%%%%%%%%%%%%%%%%%%%%%%%%
%%%%%%%%%%%%%%%%%%%%%%%%%%%%%%%%
%%%%%%%%%%%%%%%%%%%%%%%%%%%%%
%%%%%%%%%%%%%%%%%%%%%%%%%%%%
\section{Acylindrically hyperbolic groups and the proof of Theorem \ref{thm:AHregular}}\label{sec:AH}

The main result of this section is the following:
\begin{theorem}\label{thm:poison}
Let $G$ be a finitely generated group and $X$ any finite symmetric generating set. Let $Z$ be a (possibly infinite) generating set of $G$ and suppose that $H$ is quasi-isometrically embedded in $(G,Z)$, the action of $G$ on $\ga(G,Z)$ is acylindrical  and $\ga(G,Z)$ is hyperbolic.  Then $(G,X)$ has BCD relative to $H$.
\end{theorem}
We will use this result to show that any finitely generated acylindrically hyperbolic group satisfies the hypothesis of Theorem \ref{thm:crit}. We start by collecting all the required definitions. % and basic results.

An action denoted by $\circ$ of a group $G$ on a metric space $(\cS,\d)$ is called {\it acylindrical} if for every $\e> 0$ there exist $R\geq 0$ and $N \geq 0$ such that for every two points $x,y\in \cS$ with $\d(x,y)\geq R$ there are at most $N$ elements of $G$ satisfying
$$\d(x, g\circ x)\leq \e \qquad \text{ and }\qquad \d(y, g\circ y) \leq \e .$$ 

A group $G$ is called {\it acylindrically hyperbolic} (term introduced in \cite{OsinAH}) if it admits a non-elementary acylindrical action on a hyperbolic space (in this situation non-elementary is equivalent to $G$ being non-virtually cyclic and the action having unbounded orbits).

The following is a modification of \cite[Lemma 3.4]{SistoZ}, and we include the proof for completeness.

\begin{lemma}\label{lem:Xseparation}
Let $G$ be generated by a finite set $X$.
Suppose that $G$ acts acylindrically on a hyperbolic Cayley graph $\ga(G,Z)$.

Then for every $\alpha>0$  there exists
$R$ such that the following hold:
for every $D$ there is $B$ so that for all $x,y,x_1,y_1\in G$ with $\d_Z(x,y)\geq R$, $\d_Z(x,x_1)\leq \alpha, \d_Z(y,y_1) \leq \alpha$ and  $\d_X(x_1,y_1)\leq D$
the inequality $\d_X(x,x_1)\leq B$ is true.
\end{lemma} 
\begin{proof}
We can assume without loss of generality that $x=1$.

Since the action of $G$ on $\ga(G,Z)$ is acylindrical,
given $\alpha>0$ there exist $R_0$ and $N$ such that for all
$z,t\in G$, $\d_Z(z,t)\geq R_0$ the following holds 
\begin{equation}\label{eq:acyl2}
|\{h\in G \mid \d_Z(z,hz)\leq 2\alpha \text{ and } \d_Z(t,ht)\leq 2\alpha\}|<N.
\end{equation}

We let $R=R_0+2\alpha$. Fix $y\in G$ such that $\d_Z(1,y)>R$. For $x_i\in G$ such that $\d_Z(1,x_i)\leq \alpha$ consider the set $$S_{x_i}=\{g\in G \mid    \d_Z(y,x_ig)\leq \alpha\}.$$
Suppose that $g\in S_{x_1}\cap S_{x_2}$,
then $$\d_Z(g,x_1^{-1}x_2g)=\d_Z(x_1g,x_2g)\leq \d_Z(y, x_1g)+\d(y,x_2g)\leq 2\alpha$$ and $$\d_Z(1, x_1^{-1}x_2) = \d_Z(x_1,x_2)\leq \d_Z(x_1,1)+\d_Z(1,x_2)\leq 2\alpha.$$
Note that $\d_Z(1,g)=\d_Z(x_1,x_1g)\geq \d_Z(1,y)-\d_Z(1,x_1)-\d_Z(y,x_1g)>R-2\alpha =R_0$.
Therefore, by \eqref{eq:acyl2}, for a fixed $g$ there are at most $N$
different $x\in G$ such that $\d_Z(1,x)\leq \alpha$ and $g\in S_x$.
Thus for any $D>0$ we have that
$$|T|\leq N \cdot |\ball_X(D)| \;\text{ where }\;T= \bigcup_{g\in \ball_X(D)}  \{x\in G \mid \d_Z(1,x)\leq \alpha, g\in S_x\}.$$

The above inequality shows that $T$ is a finite set, so there is a $B>0$ be such that $T\subseteq \ball_X(B)$, which implies that for all $x\in T$ we have $\d_X(1,x)\leq B$.  This completes the proof.
\end{proof}

Recall that geodesics in hyperbolic spaces diverge exponentially.
\begin{lemma}\cite[III.H.Proposition 1.6]{BridsonHaefliger}\label{lem:avoid} Let $\ga(G,Z)$ be a $\delta$-hyperbolic geodesic space. If a path $p$ between points on a geodesic $q$ avoids some ball $\ball_Z(D)$ centered at any vertex $v$ on the
geodesic, then its length satisfies $\ell(p) \geq 2^{\frac{D-1}{\delta}}
$.
\end{lemma}

The next lemma shows that in a hyperbolic space if a path $q$ with the same end points as a geodesic $p$ has length linearly bounded by the length of $p$, then at least a fixed proportion of points in $q$ are uniformly close to $p$.

\begin{lemma}\label{lem:linearlength->close2geo}
Suppose that $G$ admits a presentation $\prs{Z}{S}$ whose linear isoperimetric inequality has constant $L$, and where all the relators have length at most $M$.

Then for every $K$ there exist $r$ and $0< \e \leq 1$ such that the following hold:
if $p$ is a geodesic in $\ga(G,Z)$, $p_0$ a subpath of $p$ of length $\ell(p_0)>2r$, and $q$ a path in $\ga(G,Z)$
 with the same endpoints as $p$ such that $\ell(q)\leq K\ell(p_0)$, 
then there are at least $\e \ell(p_0)$ vertices in $p_0$ at distance $\leq r$ from  points in $q$.
\end{lemma}
\begin{proof}
It is a classical result (see \cite[Lemma 4.9]{DGO} for example) that the  Cayley graph $\ga(G,Z)$ of a group admitting a presentation with bounded length relators and linear isoperimetric function is $\delta$-hyperbolic for some $\delta>0$. 
Let $\delta>0$ be the hyperbolicity constant of $\ga(G,Z)$.

Let $D$ be a number satisfying 
\begin{equation}\label{eq:Davoid}
2^{\frac{D-1}{\delta}} > 2(3LKM)(2D+4M)
\end{equation}

We are going to show that we can take $r= D+2M$.

Consider a van Kampen diagram $\Delta$ of minimal area whose boundary  is $pq^{-1}$, as in Figure 1 (we refer to \cite[I.8A.4]{BridsonHaefliger} for the definition of a van Kampen diagram). 
We claim that if $v$ is a vertex on $p$  such that $\d(v,q)>D+2M$,
then there exists a path $s_v$ in $\Delta$ whose end points are the vertices of $p$ at distance $D+M$ from $v$ and such that for all $u\in s_v$, $D<\d(u,v)<D+2M$.
To prove the claim,  consider the subcomplex  $\Delta_v$ spanned by all $2$-cells of $\Delta$ whose vertices $u$ satisfy that $D<\d_{\Delta}(u,v)<D+2M$.
We linearly extend the graph metric on $\Delta_v$ to the whole complex and construct a continuous map $\phi \colon \Delta \to \R$ assigning to each point its distance to $v$. Since $\d(v,q)>D+2M$ and $\Delta$ is homeomophic to a disk in $\R^2$,  $\phi^{-1}(D+M)$ is connected,  contained in $\Delta_v$, and contains the two vertices $v_-$ and $v_+$ of $p$ at distance $D+M$ of $v$. 
Therefore, there is a path in $\phi^{-1}(D+M)\subseteq \Delta_v$, with end points in $p$ at distance $2(D+M)$ from each other. 
We can deform this path into a   path $s_v$ in the $1$-skeleton of $\Delta_v$. This proves the claim.

\begin{figure}[h!]\label{vanKampenpic}
\labellist
		\small \hair 2pt
		\pinlabel $p$ at 24 -8
		\pinlabel $q$ at 24 75
		\pinlabel $v_i$ at 120 -8
		\pinlabel $v_j$ at 305 -8
		\pinlabel $s_{v_i}$ at 64 35
		\pinlabel $s_{v_j}$ at 245 35
		\pinlabel $>D+2M$ at 130 80
		\pinlabel $>D+2M$ at 360 80
		\endlabellist
\begin{center}
\includegraphics[scale=0.7]{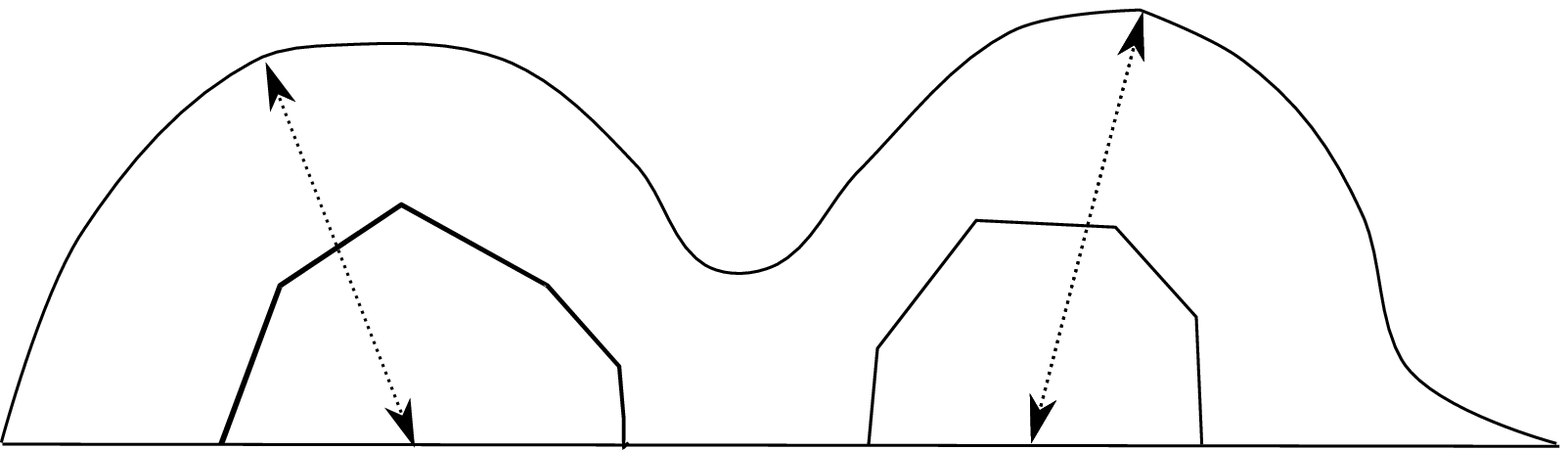}
\end{center}
\caption{Van Kampen diagram}
\end{figure}

Let $n=\lfloor \frac{\ell(p_0)}{2D+4M}\rfloor$ and let $v_1,\dots, v_n$ be vertices in $\ell(p_0)$ such that $\d(v_i,v_j)\geq 2(D+2M)$ for $i\neq j$. See Figure 1.
Note that if $\d(v_i,q)>D+2M$, then by the claim there exists a path $s_{v_i}$  in $\Delta$ with end points in $p$ and avoiding a ball of radius $D$ centered at $v_i$. Moreover, since $\d(v_i,v_j)\geq 2D+4M$ if $i\neq j$ and $\d(v_k,s_k)\leq D+2M$ for all $k$, then $s_{v_i}\cap s_{v_j}=\emptyset.$

Now suppose that for some $0\leq \nu \leq 1$ there are  $\lfloor \nu n\rfloor$ elements of $\{v_1,\dots,v_n\}$ such that $\d(v_i,q)>D+2M$.
By Lemma \ref{lem:avoid} we get that
there are at least $ \lfloor \nu n \rfloor 2^{\frac{D-1}{\delta}}$ different vertices in $\Delta$. Since $\ell(pq^{-1})\leq 2\ell(q)\leq 2K\ell(p_0)$ we have 
$$\left(\nu \left(\dfrac{\ell(p_0)}{2D+4M}-1\right)-1\right)2^{\frac{D-1}{\delta}}\leq \lfloor \nu n\rfloor 2^{\frac{D-1}{\delta}} \leq 2L K\ell(p_0) M,$$
where the last inequality is an upper bound for the total number of vertices in the van Kampen diagram $\Delta$.

Therefore we get that
$$\nu\leq \dfrac{ 2D+4M}{2^{\frac{D-1}{\delta}}} \cdot \dfrac{(2LK\ell(p_0)M+1)}{\ell(p_0)- (2D+4M)}\leq\dfrac{ (2D+4M)(3LKM)}{2^{\frac{D-1}{\delta}}} \cdot \dfrac{\ell(p_0)}{\ell(p_0)- (2D+4M)} .$$
Hence if $\ell(p_0)>2(2D+4M)$, we have that $\ell(p_0)/(\ell(p_0)-(2D+4M))<2$ and by \eqref{eq:Davoid} we get that
$$0\leq  \nu \leq  \e_0 \coloneq \dfrac{ (2D+4M)(3LKM)}{2^{\frac{D-1}{\delta}}} 2< 1.$$
This completes the proof, since we have shown that $\nu \leq \e_0<1$ and hence there are at least $(1-\e_0)n$ elements of $\{v_1,\dots,v_n\}$ such that $\d(v_i,q)\leq D+2M$. Finally let $\e\coloneq 1-\e_0$.
\end{proof}

\begin{proof}[Proof of Theorem \ref{thm:poison}]
Let $Z$  be a generating of $G$ such that the Cayley graph $\ga(G,Z)$ is hyperbolic, the action of $G$ on $\ga(G,Z)$ is acylindrical and $H$ is quasi-isometrically embedded in  $(G,Z)$. Since the symmetric difference between $Z$ and $Z\cup X$ is finite, we have that $(G,Z)$ and $(G,Z\cup X)$ are quasi-isometric, and therefore, we can assume without loss of generality that $X\subseteq Z$.

%Without loss of generality, we can assume that $X\subseteq Z$.
Since $\ga(G,Z)$ is hyperbolic, $G$ has a presentation $\prs{Z}{S}$ with bounded length relators and linear isoperimetric inequality.

Let $Y$ be any finite generating set of $H$. We identify each element of $Y$ with a geodesic word over $X$ representing the same element, and for $V\in Y^*$ we write $\wh{V}$ to denote the word $V$ viewed as a word over $X$ through this fixed identification. 
As $H$ is quasi-isometrically embedded in $G$, there exist $\lambda\geq 1$ and $c\geq 0$ such that for any  $V\in \geocpl(H,Y)$,
$\wh{V}$ labels a cyclic $(\lambda,c)$-quasi-geodesic in $\ga(G,Z)$. Since $\ga(G,Z)$ is hyperbolic, it satisfies the BCD property.
Let $D_0=D_0(\lambda,c)$ be the BCD constant for cyclic $(\lambda, c)$-quasi-geodesics.

 Let $U\in \geocl(G,X)$ and $V\in \geocpl(H,Y)$ be conjugate, and suppose that $C$ is a minimal $Z$-length conjugator, up 
to cyclic permutations of $U$ and $\wh{V}$; we can assume without loss of generality that $CU =_G \wh{V}C$. Note that by definition $\ell_X(U)\leq \ell_X(\wh{V})$.

The goal of the proof is to use the BCD property in the hyperbolic $\ga(G,Z)$, together with Lemmas \ref{lem:linearlength->close2geo} and \ref{lem:Xseparation}, to find a short conjugator of $U$ and $V$ not only with respect to $Z$, but with respect to $X$.
Take $W\in \geol(G,Z)$ representing the same element as $U$, 
and consider the 4-gon $prq^{-1}s^{-1}$ in $\ga(G,Z)$ where $\Lab(p)\equiv \wh{V}$, 
$\Lab(r)\equiv \Lab(s)\equiv C$, $\Lab(q)\equiv W$. By hyperbolicity and stability of quasi-geodesics, there is a constant $K_1=K_1(\delta, \lambda,c)$ so that each side of the 4-gon is at distance $K_1$ of the other three. Let $q_1$ be the path labelled by $U$ with same end points as $q$.
See Figure \ref{fig:diagramH4}.

\begin{figure}[h!]
\vspace{0.2cm}
		\labellist
		\small \hair 2pt
		\pinlabel $1$ at 24 0
		\pinlabel $p,\,\text{Lab}(p)\equiv \widehat{V},\text{ a }(\lambda,c)\text{-q.g. in }(G,Z)$ at 250 -4 
		\pinlabel $q_1,\,\text{Lab}(q_1)\equiv U,\text{ a path satisfying }\ell(U)\leq \ell(\widehat{V})$ at 250 200 
		\pinlabel ${u}$ at 55 87
		\pinlabel ${v}$ at 427 87
		\pinlabel ${q},\text{Lab}(q)\equiv W,\text{ geodesic path in }(G,Z)$ at 230 96
		\pinlabel ${q_0}\text{ subpath of $q$ from $u$ to $v$}$ at 230 76
		\pinlabel $s$ at 40 75
		\pinlabel $r$ at 442 75
		\pinlabel $\geq \frac{\ell(\wh{W})}{\lambda}-c-2D_0$ at 265 50
		\pinlabel $\leq R_1$ at 116 102
		\pinlabel $\leq R_1$ at 341 102
		\pinlabel $\leq R_1$ at 396 102
		\pinlabel $\leq K_2$ at 116 32
		\pinlabel $\leq K_2$ at 341 32
		\pinlabel $\leq K_2$ at 396 32
		\pinlabel $\leq K_2$ at 224 32
		\pinlabel $x_1$ at 96 4
		\pinlabel $v_1$ at 96 78
		\pinlabel $y_1$ at 100 120
		\pinlabel $x_i$ at 321 8
		\pinlabel $v_i$ at 321 78
		\pinlabel $y_i$ at 327 126
		\pinlabel $x_n$ at 376 -2
		\pinlabel $v_n$ at 374 78
		\pinlabel $y_n$ at 383 126
		\endlabellist
%\begin{center}
\includegraphics[scale=0.95]{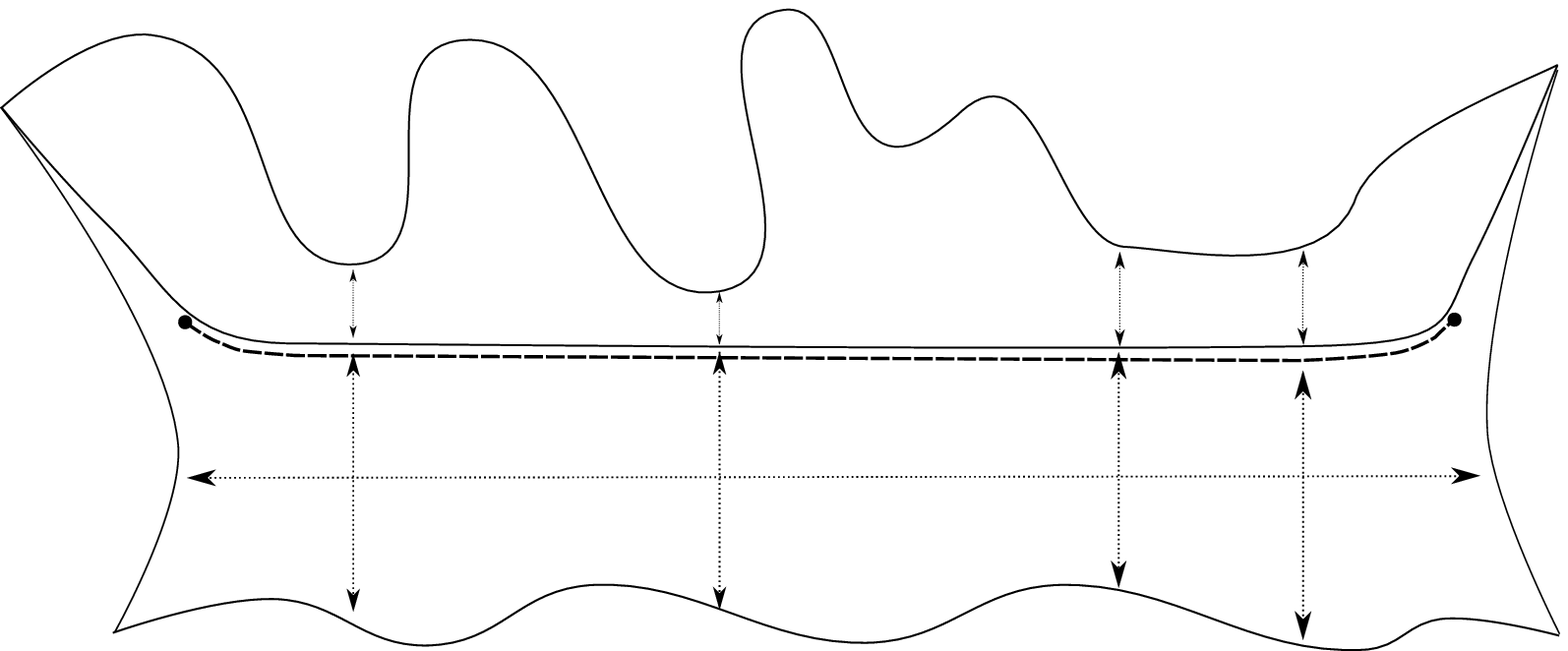}
%\end{center}
\caption{Diagram of the paths $U$ and $\wh{V}$ in $\ga(G,Z)$}
\label{fig:diagramH4}
\vspace{0.2cm}
\end{figure}

%In the remaining paragraphs we use $\ell:=\ell_Z$. 
The following argument shows how to find a `long enough' subpath $q_0$ of $q$, with respect to $q_1$, so that we can apply Lemma \ref{lem:linearlength->close2geo}.
By the BCD property, any word $T\in Z^*$ conjugate to $\wh{V}$ has length
$$\ell(T)\geq (\ell(\wh{V})/\lambda -c)-2D_0.$$
Thus, for all $t$,
$$\d_Z(r(t), s(t))\geq (\ell(\wh{V})/\lambda -c)-2D_0.$$
Suppose that 
$$(\ell(\wh{V})/\lambda -c)-2D_0> 2K_1$$
and that a vertex $r(t_1)$ of $r$ is at $Z$-distance  at most $K_1$ of a vertex $s(t_2)$ of $s$.
Then $|t_1-t_2|\leq K_1$, since otherwise we could shorten $C$; on the other side, 
this contradicts  $\d_Z(r(t_1),s(t_1))\geq  (\ell(\wh{V})/\lambda -c)-2D_0$. This shows that $d_Z(r(t_1), s(t_2)) \geq K_1$ for all $t_1, t_2$.
Also, a similar argument shows that any vertex $r(t)$ with $t>K_1$  cannot be at distance $\leq K_1$ to a vertex of $p$, since then we can change $\wh{V}$ by a cyclic permutation and reduce the length of the conjugator. This would contradict the minimality of $C$.

Thus $r(K_1+1)$ and $s(K_1+1)$ are at distance smaller than $K_1$ from some vertices $u$ and $v$ in $q$, respectively. Let $q_0$ be the subpath of $q$ from $u$ to $v$.
As $\d_Z(r(K_1+1),s(K_1+1)) \geq( \ell(\wh{V})/\lambda -c)-2D_0,$ 
we get $\d_Z(u,v)\geq ( \ell(\wh{V})/\lambda -c)-2D_0-2K_1$. 

Then $\ell(q_0) =\d_Z(u,v) \geq  \ell(\wh{V})/\lambda -(c+2D_0+2K_1)$ and since $\ell(\wh{V})\geq \ell(U)$, we get that
$\ell(q_0)\geq \ell(U)/\lambda -(c+2D_0+2K_1)$. Then
$\ell(U)$ is linearly bounded by $\ell(q_0)$;  one can for example take $\lambda'=\lambda +c+2D_0+2K_1$ and hence $\lambda'\ell(q_0)\geq \ell(U)$.
We now can apply  Lemma \ref{lem:linearlength->close2geo}, so 
there are $R_1$ and $0<\e\leq 1$ depending on $\lambda'$ such that if $\ell(q_0)>2R_1$ and $\ell(q_1)\leq \lambda' \ell(q_0)$, then at least $\e \ell(q_0)$ vertices of $q_0$ 
are at $Z$-distance at most $R_1$ from a vertex of $q_1$.

By hyperbolicity, there is some constant $K_2=K_2(K_1,\lambda,c)$ such that  
$q_0$ is in the $K_2$-neighbourhood of $p$.
Let $R$ be the constant of Lemma \ref{lem:Xseparation} with $\alpha=K_2+R_1$.
Since $q_0$ is geodesic, it follows that  there are $n=\lfloor \frac{\e_1 \ell(q_0)}{R}\rfloor$ different vertices $v_1,\dots,v_n$ of $q_0$, such that for each $v_i$ there are vertices $x_i\in p$ and $y_i\in q_1$ with  
$\d_Z(x_i,v_i)\leq K_2$, $\d_Z(v_i,y_i)\leq R_1$ and $\d_Z(x_i,x_j)>R$ for $i\neq j$. After relabelling the vertices, if necessary, we can assume that for $i<j$, $y_i$ appears before $y_j$ when traveling in $q_1$ from $(q_1)_-$ to $(q_1)_+$.
Recall that $U$ labels a geodesic in $X$ with $\ell_X(U) \leq \ell_X (\wh{V})$;
therefore, there is some $i$ such that
$\d_X(y_i,y_{i+1})\leq \ell(\wh{V})/n$. 
Since $n$ depends linearly on $\ell(\wh{V})$ there is some $D$ independent of $\wh{V}$
 such that $\ell(\wh{V})/n\leq D$. 
 We are going to use Lemma \ref{lem:Xseparation} for the four-point set $\{x_i,y_i, y_{i+1}, x_{i+1}\}$. So far we have obtained $$\d_X(y_i,y_{i+1})\leq D, \d_Z(x_i,y_i)\leq \alpha,\d_Z(x_{i+1},y_{i+1})\leq \alpha, \textrm{and} \ \d_Z(x_i,x_{i+1})>R.$$
Let $B$ be the constant of Lemma \ref{lem:Xseparation} corresponding to $D$. Then $\d_X(x_i,y_i)
 \leq B$. 
 
We have found cyclic permutations of $U$ and $\wh{V}$
 that are conjugate by an element of $X$-length less than $ B$.
 Let $M=\max\{|y|_X \mid y \in Y\}$. Then we can find a cyclic permutation
 of $U$ and $V$ conjugated by an element of $X$-length less than $B+M$. This finishes the proof.
\end{proof}

\subsection{Proof of Theorem \ref{thm:AHregular}}

It was proved in \cite[Theorem 1.4]{OsinAH} that a group is acylindrically hyperbolic 
if and only if it has a non-degenerate hyperbolically embedded subgroup in the sense of Dahmani, Guirardel and Osin \cite{DGO}:

\begin{definition}
Let $G$ be a group, $\Lambda$ a set,
$\{H _\lambda \}_{\lambda \in \Lambda}$ a collection of  subgroups of $G$ and
$Z$ a  subset of G (not necessarily finite). 

The set $Z$ is a {\it generating set relative to} $\{H _\lambda \}_{\lambda \in \Lambda}$ 
if the natural homomorphism from
$$ F = (*_{\lambda \in \Lambda}H_\lambda )* F (Z)$$
to $G$ is surjective, where $ F (Z)$ is the free group with  basis $Z.$ 
Assume that $Z$ is a generating set relative to $\{H_\lambda \}_{\lambda \in \Lambda}$ and let $R$ 
be a subset of $F$ whose normal closure is the kernel of the natural map $F\to G$.
Let $$\cH = \bigsqcup_{\lambda \in \Lambda}  H_\lambda,$$ let $S_\lambda$ be all the words in $H_\lambda$ representing the identity, and let $S=\cup S_\lambda$.
We then say that $G$ has {\it presentation relative to $Z$}
$$\prs{Z, \cH}{S, R}.$$

The presentation is {\it strongly bounded} if all words in $R$ have bounded length,
and for each $\lambda\in \Lambda$, there are only finitely many elements of $H_\lambda$
appearing  as syllables of words in $R$.

Given a word $W$ over the alphabet $Z\cup \cH$ that represents $1$ in $G$
there exists an expression for $W$ in $F$ of the form
\begin{equation}
\label{eq:defarea}
 W=_F \prod_{i=1}^n f_ir_i^{\e_1} f_i^{-1},
\end{equation}
where $r_i\in R$, $f_i\in F$ and $\e_i=\pm 1$ for $i=1,\dots, n$.
The smallest possible number $n$ in an expression of type \eqref{eq:defarea} is
called the {\it relative area of $W$} and is denoted by $\mathrm{Area}_{rel}(W)$.

A family of subgroups $\{H_\lambda\}_{\lambda \in \Lambda}$ of $G$ is 
{\it hyperbolically embedded in $(G,Z)$}, denoted $\{H_\lambda\}_{\lambda \in \Lambda}\h (G,Z)$, if 
there exists a strongly bounded presentation relative to $\{H_\lambda\}_{\lambda \in \Lambda}$ and $Z$,
and there is a constant $L\geq 0$, called an {\it isoperimetric constant},
such that 
$$\mathrm{Area}_{rel}(W)\leq L \ell_{Z\sqcup \cH}(W)$$
for all words $W$ over $Z\sqcup \cH$ that are the identity in $G$.
In particular, $\ga(G, Z\sqcup \cH)$ is Gromov hyperbolic (see \cite[Lemma  4.9]{DGO}).

We say that $\{H_\lambda\}_{\lambda \in \Lambda}$ is {\it hyperbolically embedded} in $G$,
denoted $\{H_\lambda\}_{\lambda \in \Lambda}\h G$, 
if there is some $Z$ such that $\{H_\lambda\}_{\lambda \in \Lambda}\h(G,Z)$.
We note that $G\h G$ (take $Z$ to be empty) and for any collection of finite subgroups 
$\{H_\lambda\}_{\lambda \in \Lambda}$,  $\{H_\lambda\}_{\lambda \in \Lambda}\h G$ (take $Z$
to be $G$). These are the {\it degenerate} cases.
\end{definition}

Hyperbolically embedded subgroups satisfy the properties in Lemma \ref{lem:AHfacts} (see \cite[Lemmas 3.1, 3.2]{AMS}). Furthermore, it was shown in \cite[Theorem 4.42]{DGO} or \cite[Theorem 3.9]{AMS} (which has a simpler proof when $\ga(G,Z)$ is hyperbolic) that the converse also holds, i.e. if (1) and (2) below hold, then $H\h (G,Z)$.

\begin{lemma}\label{lem:AHfacts}
Suppose that $H\h (G,Z)$. Then
\begin{enumerate}
\item[{\rm (1)}] $H$ is finitely generated and quasi-isometrically embedded in $(G,Z)$;
\item[{\rm (2)}] for every $r>0$ there exists $R$ such that for $g\in G$, if $\diam(H \cap \nei{gH}{r}{Z})>R$ then $g\in H$.
\end{enumerate}
Here $\nei{S}{r}{Z}$ denotes the $r$-neighbourhood of $S$ with respect to the metric $\d_Z$.
\end{lemma}

%We will need the next two results.
%
%\begin{lemma}[{\cite[Cor. 4.27]{DGO}}] \label{lem:changegenset} Suppose that $G$ is a group, $\{H_\lambda\}_{\lambda \in \Lambda}$ is a family of subgroups of $G$ and $Z_1,Z_2 \subseteq G$ are
%relative generating sets of $G$ with respect to $\{H_\lambda\}_{\lambda \in \Lambda}$ such that $|Z_1\bigtriangleup Z_2|<\infty$. Then
%$\{H_\lambda\}_{\lambda \in \Lambda} \h (G,Z_1)$ if and only if  $\{H_\lambda\}_{\lambda \in \Lambda} \h (G,Z_2).$
%\end{lemma}
%
\begin{theorem}\cite[Theorem 5.4]{OsinAH}\label{thm:acyl}
Let $G$ be a group, $\{H_\lambda\}_{\lambda \in \Lambda}$ a finite collection of subgroups of $G$, $Z$ a subset
of $G$ such that $\{H_\lambda\}_{\lambda \in \Lambda}\h (G,Z)$. Then there exists $Y \subseteq G$ such that $Z \subseteq Y$ 
and the
following hold.
\begin{enumerate}
\item[{\rm (a)}] $\{H_\lambda\}_{\lambda \in \Lambda}\h (G,Y)$.
\item[{\rm (b)}] The action of $G$ on $\ga(G, Y\sqcup (\sqcup_{\lambda \in \Lambda}H_\lambda))$ is acylindrical.
\end{enumerate}
\end{theorem}

\begin{proof}[Proof of Theorem \ref{thm:AHregular}]
Let $X$ be the fixed generating set of $G$. Since $G$ is acylindrically hyperbolic, it contains non-degenerate hyperbolically embedded subgroups.

By \cite[Theorem 6.8]{DGO}, there is a family of virtually cyclic subgroups $\{H_1,H_2,H_3\}$
and $Z_0\subseteq G$ such that $\{H_1,H_2, H_3\}\h(G,Z_0)$.
By Theorem \ref{thm:acyl}, there is $Z_1$ such that $X\subseteq Z_0 \subseteq  Z_1$,
 $\{H_1,H_2, H_3\}\h(G,Z_1)$ and the action of $G$ on  $\ga(G,Z_1\sqcup H_1 \sqcup H_2 \sqcup H_3)$
is acylindrical.

In the proof of \cite[Theorem 6.14]{DGO} it is shown that
starting from $\{H_1,H_2, H_3\}\h(G,Z)$ one can construct  $H$, a non-elementary virtually free subgroup  of $G$,
such that $H\h (G, Z_1\sqcup H_1\sqcup H_2 \sqcup H_3)$.  Set $Z= Z_1\sqcup H_1\sqcup H_2 \sqcup H_3$.

We need to establish that all the hypotheses of Theorem \ref{thm:crit} hold.
By Theorem \ref{thm:poison}, $(G,X)$ has BCD relative to $H$, and hence it is sufficient to show that $H$ is almost malnormal and Morse. However, these are known properties for hyperbolically embedded subgroups. 

Indeed,  if $h\in H$ has infinite order and $g^{-1}hg\in H$ for some $g\in G$,
then $H\cap \nei{gH}{|g|_Z}{Z}$ has infinite diameter and by
Lemma \ref{lem:AHfacts}, $g\in H$. Therefore $H$ is almost malnormal. By \cite[Theorem 2]{SistoZ}, hyperbolically embedded subgroups are Morse.
\end{proof}

%%%%%%%%%%%%%%%%%%%%%%%%%%%%%%%%%%%%%%%%%%%%%%%%%%%%%%%%%%%%%%%%%%%%%%%%%%%%
%%%%%%%%%%%%%%%%%%%%%%%%%%%%%%%%%%%%%%%%%%%%%%%%%%%%%%%%%%%%%%%%%%%%%%%%%%%%
\section*{Acknowledgments}
%%%%%%%%%%%%%%%%%%%%%%%%%%%%%%%%%%%%%%%%%%%%%%%%%%%%%%%%%%%%%%%%%%%%%%%%%%%%
%%%%%%%%%%%%%%%%%%%%%%%%%%%%%%%%%%%%%%%%%%%%%%%%%%%%%%%%%%%%%%%%%%%%%%%%%%%%

The authors would like to thank Michael Coons, Murray Elder and Denis Osin for helpful discussions, and thank Derek Holt for his valuable comments on Section 5.

The first author was supported by the MCI 
(Spain) through project MTM2014-54896-P. 
Both authors were supported by the Swiss National Science 
Foundation grant Professorship FN PP00P2-144681/1.  

%%%%%%%%%%%%%%%%%%%%%%%%%%%%%%%%%%%%%%%%%%%%%%%%%%%%%%%%%%%%%%%%%%%%%%%%%%%%

\bigskip

\textsc{Yago Antol\'{i}n,
Vanderbilt University,
Department of Mathematics,
1326 Stevenson Center,
Nashville, TN 37240, USA.}

\emph{E-mail address}{:\;\;}\texttt{yago.anpi@gmail.com}

\bigskip

\textsc{Laura Ciobanu,
Mathematics Department,
University of Neuch\^atel,
Rue Emile-Argand 11,
CH-2000 Neuch\^atel, Switzerland
}

\emph{E-mail address}{:\;\;}\texttt{laura.ciobanu@unine.ch}

\end{document}